\documentclass[12pt,twoside,reqno]{amsart}
\usepackage{amsmath}
\usepackage{amssymb}
\usepackage{wasysym}
\usepackage{psfrag}
\usepackage{graphicx,epsf,amsmath}  
\usepackage{epsf,graphicx}
\usepackage{comment}

\usepackage{hyperref}

\newtheorem{lem}{Lemma}[section]
\newtheorem{theorem}[lem]{Theorem}
\newtheorem{propo}[lem]{Proposition}
\newtheorem{defi}[lem]{Definition}
\newtheorem{coro}[lem]{Corollary}
\newtheorem{rema}[lem]{Remark}

\numberwithin{equation}{section}

\def\RR{{\mathbb R}}

\def\dfrac{\displaystyle\frac}

\def\R{\mathbb{R}}
\def\epsilon{\varepsilon}

\title[Hamiltonian estimates for fractional Laplacians]
{Nonlinear equations for fractional Laplacians I: Regularity, maximum principles,
and Hamiltonian estimates}
\author{Xavier Cabr\'e}
\author{Yannick Sire}

\thanks{The first author was supported by
grants MTM2008-06349-C03-01 (Spain) and 2009SGR-345 (Catalunya)}
\address{X. Cabr\'e: ICREA and Universitat Polit{\`e}cnica de Catalunya,
Departament de Matem{\`a}tica Aplicada I, Diagonal 647, 08028
Barcelona, Spain}
\email{xavier.cabre@upc.edu}
\address{Y. Sire: Universit\'e Paul C\'ezanne, LATP,
Facult\'e des Sciences et Techniques, Case cour A,
Avenue Escadrille Normandie-Niemen, F-13397 Marseille Cedex 20,
France, and CNRS, LATP, CMI, 39 rue F. Joliot-Curie, F-13453 Marseille Cedex
13, France}
\email{sire@cmi.univ-mrs.fr}

\begin{document}

\begin{abstract}
This is the first of two articles dealing with the equation 
$(-\Delta)^{s} v= f(v)$ in $\mathbb{R}^{n}$,
with $s\in (0,1)$, where $(-\Delta)^{s}$ stands for the fractional
Laplacian ---the infinitesimal generator of a L\'evy process. This 
equation can be realized as a local linear degenerate elliptic 
equation in $\mathbb{R}^{n+1}_+$ together with a nonlinear Neumann boundary 
condition on $\partial \mathbb{R}^{n+1}_+=\mathbb{R}^{n}$. 

In this first article, we establish necessary conditions on the nonlinearity $f$
to admit certain type of solutions, with special interest in bounded increasing solutions 
in all of $\mathbb{R}$. These necessary conditions (which will be proven in a follow-up paper to be also
sufficient for the existence of a bounded increasing solution) are derived
from an equality and an estimate involving a Hamiltonian
---in the spirit of a result of Modica for the Laplacian.
In addition, we study regularity issues, as well as maximum and Harnack principles 
associated to the equation.
\end{abstract}
\maketitle


\section{Introduction}

This paper is devoted to the study of the nonlinear problem 
\begin{equation}
\label{problem}
(-\Delta)^{s} v= f(v)\,\,\,\mbox{ in } \mathbb{R}^{n},
\end{equation}
where $s\in (0,1)$ and
\begin{equation}\label{fracpv}
(-\Delta)^{s} v(x)=C_{n,s}\ \textrm{P.V.} \int_{\mathbb{R}^{n}}
\frac{v(x)-v(\overline x)}{|x-\overline x|^{n+2s}}\,d\overline x. 
\end{equation}
Here $\textrm{P.V.}$ stands for the Cauchy principal value, and $C_{n,s}$ is
a positive constant depending only on $n$ and $s$ 
---whose value is given in Remark~\ref{constant} below. 
The above integral is well defined if, for instance, $v$ is bounded (which ensures
the integrability at infinity) and $v$ is $C^2_{\rm loc}(\RR^{n})$ (which ensures
the integrability at $\overline x=x$ in the principal value sense). 

As explained in more detail in section~3 below, up to an explicit multiplicative constant
(given in Remark~\ref{constant})
in front of the nonlinearity $f$, problem \eqref{problem} can be realized in a
local manner through the nonlinear boundary value problem 
\begin{equation}
\label{extAlpha}
\begin{cases}
\textrm{div\,} (y^{a}\,\nabla u)=0&\text{ in } \mathbb{R}^{n+1}_+\\ 
(1+a) \displaystyle{\frac{\partial u}{\partial{\nu}^{a}}}
=f(u) &\text{ on } \partial\mathbb{R}^{n+1}_+=\RR^n ,
\end{cases}
\end{equation}
where $n\geq 1$, $\mathbb{R}^{n+1}_+=\{(x,y) \in \mathbb{R}^n \times \mathbb{R} : y>0\}$
is the halfspace, $\partial\mathbb{R}^{n+1}_+=\{y=0\}$,
$u=u(x,y)$ is real valued, and 
$$
\frac{\partial u}{\partial \nu^a}=-\lim_{y\downarrow 0} y^{a} \partial_y u
$$
is the conormal exterior derivative of~$u$.
Points in $\mathbb{R}^{n}$ are denoted by 
$x=(x_1,\ldots ,x_{n})$. The parameter $a$ belongs to $(-1,1)$
and is related to the power $s$ of the fractional
Laplacian $(-\Delta)^s$ by the formula 
$$a=1-2s \in (-1,1)$$
---a relation that we assume throughout the paper.

Indeed, Caffarelli and Silvestre \cite{cafS} have proved the following formula 
relating the fractional
Laplacian $(-\Delta)^s$ to the Dirichlet to Neumann operator for \eqref{extAlpha}:
\begin{equation}\label{cttNeumann}
(-\Delta)^{s} \{u(\cdot , 0)\}=d_s\frac{\partial u}{\partial
  \nu^{a}} \qquad\text{in } \RR^n={\partial \mathbb{R}^{n+1}_+},
\end{equation}
where $d_s$ is a positive constant depending only on~$s$
(see Remark~\ref{constant} below
for  its value).

The aim of this paper ---and of the forthcoming one \cite{CS2}--- is to study 
two types of bounded solutions of~\eqref{problem}:
\begin{enumerate}
\item[(a)] Solutions $v=v(x)$ of~\eqref{problem} which
are monotone increasing, say from $-1$ to $1$, 
in one of the $x$-variables. These solutions are named {\em layer solutions} and 
constitute our main interest.
\item[(b)]  Radial solutions $v=v(|x|)$ of~\eqref{problem} tending, say, to $0$
as $|x| \rightarrow  \pm \infty$.
\end{enumerate}

In the second part \cite{CS2} of this work, we will be concerned with the existence, uniqueness,
symmetry and variational properties, as well as the asymptotic
behavior of layer solutions. These questions will be related, of course, 
to structural assumptions made on the nonlinearity~$f$.

In this first article, we establish necessary conditions on the nonlinearity $f$
to admit a layer solution in $\RR$ (i.e., in dimension one), and also to admit radial
solutions in $\RR^n$ with limit at infinity. In the case of
layer solutions, our necessary conditions \eqref{nec1} and \eqref{nec2} below will be proven 
in \cite{CS2} to be also
sufficient for the existence of a layer solution.
Our necessary conditions are derived
from a new equality and a new estimate involving the Hamiltonian for \eqref{extAlpha} 
---in the spirit of a celebrated inequality of Modica~\cite{M1} for the Laplacian.
In subsection~1.1 below we explain the Hamiltonian structure of problem~\eqref{extAlpha}.
Let us also recall that
Modica proved that every bounded solution of $\Delta v -G'(v)=0$ in the whole $\RR^n$
satisfies $(1/2)|\nabla v|^2\leq G(v)$ pointwise in all $\RR^n$, assuming only
that $G\geq 0$ in $\RR$. 
Here we prove an analogue of such an estimate in dimension one in
the case of nonlocal operators given by the fractional Laplacians. 
This is done via a careful study of the local boundary value
problem~\eqref{extAlpha}. 

In addition, in this first paper we also study regularity issues, as well as maximum, Liouville, 
and Harnack principles associated to the fractional Laplacian in $\RR^n$. These 
tools will be needed in this paper and in its follow-up.

Our work extends to all fractions $s\in (0,1)$ the results of one of the
authors and J.~Sol\`a-Morales~\cite{CSM} for the case $s=1/2$.
\cite{CSM} studies bounded (specially layer) solutions of 
\begin{equation*}
\begin{cases}
\Delta u=0&\text{ in } \mathbb{R}^{n+1}_+\\ 
\displaystyle{\frac{\partial u}{\partial{\nu}}}
=f(u) &\text{ on } \partial\mathbb{R}^{n+1}_+ ,
\end{cases}
\end{equation*}
which corresponds to the case $a=0$ in \eqref{extAlpha}. It is 
well known that the Dirichlet to Neumann operator associated to
the previous problem is precisely $(-\Delta)^{1/2}$. Therefore, layer
solutions of the previous equation are actually heteroclinic
connections (between $-1$ and $1$) of 
\begin{equation*}
(-\Delta)^{1/2} v= f(v)\,\,\,\mbox{ in } \mathbb{R}^{n},
\end{equation*}
where $v$ is the trace of $u$ on the boundary $\left \{ y= 0 \right \}.$
The goal of our papers is to generalize this study to any fractional
power $s\in (0,1)$ of the Laplacian. We will make a great use of the tools 
developed in~\cite{CSM}. However, some new difficulties arise due to the degeneracy 
of the operator in \eqref{extAlpha}.

The first equation in \eqref{extAlpha} is a linear degenerate elliptic equation with weight $y^{a}$. 
Since $a \in (-1,1)$, the weight
$y^{a}$ belongs to the Muckenhoupt class $A_2$; see~\cite{muck}. 
More precisely, a nonnegative 
function $w$ defined in $\RR^N$ is said to be $A_2$ if, for some constant $C$,
\begin{equation*}
\sup_B \, (\frac{1}{|B|} \int_B w)(\frac{1}{|B|} \int_B w^{-1}) \leq C
\end{equation*}
for all balls $B$ in $\mathbb{R}^N$. It is easy to verify that $|y|^a\in 
A_2(\RR^{n+1})$ for $a\in (-1,1)$. As we explain in section~3, this fact allows to use
the regularity results of Fabes, Jerison, Kenig, and Serapioni~\cite{FKS,FJK}
concerning divergence form equations with $A_ 2$ weights.
Another important property of the weight $y^{a}$ is that it
depends only on the extension variable $y$ and not on the tangential variables~$x$. 
Thus, \eqref{extAlpha} is invariant under translations in $x$ ---as it is equation
\eqref{problem}. In \cite{CS2}, this will allow us, for instance, to use for \eqref{extAlpha} the
sliding method of Berestycki and Nirenberg.

The study of elliptic equations involving fractional powers of the
Laplacian appears to be important in many
physical situations in which one has to consider long-range or
anomalous diffusions. From a probabilistic point of view, the
fractional Laplacian is the infinitesimal generator of a
L\'evy process ---see the book of Bertoin~\cite{B}, for instance.

A lot of interest is currently devoted to the study of nonlinear
equations involving fractional powers of the Laplacian.
This type of operator seems to have a
prevalent role in physical situations such as combustion (see
\cite{CRS}), dislocations in mechanical systems (see \cite{MI})
or in crystals (see \cite{GM,T1}). In addition, these operators arise in
modelling diffusion and transport in a highly heterogeneous medium;
they may be used in the description of the diffusion of a liquid in some
heterogeneous media, or as an effective diffusion in a limiting
advection-diffusion equation with a random velocity field;
see for instance \cite{MVV}.

\subsection{Formal Hamiltonian structure}
As in \cite{CSM}, the quantity appearing in our main results ---see for
instance \eqref{mmequal} below---
arises naturally when one looks at problem \eqref{extAlpha}
for $n=1$ as a formal Hamiltonian system in infinite dimensions.
Here the time variable is $\tau =x$, the position $q$ is the function
$u(x,\cdot)=u(\tau,\cdot )$ in the halfline $\{y\ge 0\}$, and the momentum is
$p=q'=u_x(\tau,\cdot )$. 

From the action ---that is, in PDE terminology 
the energy functional~\eqref{enerfunc} below, which we already have at hand--- 
we see that the Lagrangian is 
$L(q,p)=(1/2)\|p\|^2_{2,a}+W(q)$, with 
$$
W(q)=\frac{1}{2}\|\partial_y
q\|^2_{2,a}+\frac{1}{1+a}G(q(0))
$$
and $\|w\|^2_{2,a}=\int_0^{+\infty} y^a |w(y)|^2\,dy$.

The Legendre transform of $L$ with respect to $p$ gives
the Hamiltonian 
\begin{eqnarray*}
H(q,p)&=&\int_0^{+\infty}\dfrac{t^{a}}{2}\left\{
u_x^2(x,t)-u_y^2(x,t)\right\} dt-\frac{1}{1+a}G(u(0,x))\\
&=& \frac{1}{2}\|p\|^2_{2,a}-W(q).
\end{eqnarray*}
One can easily check that its associated Hamiltonian system
\begin{eqnarray*}
\begin{pmatrix}
q' \\p'
\end{pmatrix}
=
\begin{pmatrix}
p \\ W'(q)
\end{pmatrix}.
\end{eqnarray*}
is formally problem~\eqref{extAlpha}.

Thus, our equation admits a Hamiltonian structure in an infinite dimensional 
phase space. However, in this paper we do not address the question of setting it as
a true well posed semigroup. Note that a lot of challenging issues usually arise with
infinite-dimensional Hamiltonian systems ---see for instance~\cite{chernoff}.

\subsection{Outline of the article}
The paper is organized as follows.
Section~2 contains the statements of our main results. In section~3 we explain
the relation between problems \eqref{problem} and \eqref{extAlpha}, and we present
the Poisson kernel and a regularity result for \eqref{extAlpha}.
Section~4 contains results on the operator
$L_a$ appearing in \eqref{extAlpha}; we establish Schauder estimates, a Harnack inequality, a 
Hopf principle, 
maximum principles, and a Liouville theorem.
Section~5 is concerned with the proof of the Hamiltonian equality and estimates.
In section~6 we prove our results on layers as the fraction $s$ tends to $1$.
Finally, in section~7 we collect the
proofs of our main results, Theorems \ref{necLayers}, \ref{modthm}, \ref{necBounds} and \ref{modthmGS},
using the results established in previous sections.

\section{Main Results: Hamiltonian identity and necessary conditions
on the nonlinearity for existence}

Throughout the paper we assume that $s \in (0,1)$ and that
the nonlinearity satisfies
$$f\in C^{1,\gamma}(\RR) \quad \text {for some }
\gamma >\max(0,1-2s).
$$ 
We will denote $G$ the
associated potential, i.e., 
\begin{equation*}
G'=-f
\end{equation*}
---which is defined up to an additive constant. We recall that
the problem under study is \eqref{extAlpha}, i.e., 
\begin{equation*}
\begin{cases}
\textrm{div\,} (y^{a}\,\nabla u)=0&\text{ in } \mathbb{R}^{n+1}_+\\ 
(1+a) \displaystyle{\frac{\partial u}{\partial{\nu}^{a}}}
=f(u) &\text{ on } \partial\mathbb{R}^{n+1}_+ ,
\end{cases}
\end{equation*} 
with $a=1-2s. $ In the sequel we will denote 
$$
L_{a}w=\textrm{div\,} (y^{a}\,\nabla w).
$$ 
We use the notation
\begin{align*}
& B_R^+=\{ (x,y)\in\R^{n+1} : y>0, |(x,y)|<R\}, \\
& \Gamma_R^0=\{ (x,0)\in\partial\R^{n+1}_+ : |x|<R\}, \text{ and}\\
& \Gamma_R^+=\{ (x,y)\in\R^{n+1} : y\ge 0, |(x,y)|=R\}. 
\end{align*}

It is easy to see that  \eqref{extAlpha} has a variational
structure, corresponding to the energy functional
\begin{equation}\label{enerfunc}
E_{B_R^+}(u)=
\int_{B_R^+}\dfrac{1}{2}y^{a} |\nabla u|^2\, dx dy
+\int_{\Gamma^0_R} \frac{1}{1+a} G(u)\, dx. 
\end{equation}

We are concerned with the following types of solutions. The first class (layer solutions)
consists of solutions which are increasing and have limits at infinity 
in one Euclidean variable in the space
$\RR^n$ of $x$-variables. In the following definition, and for future
convenience, after a rotation we may assume that such variable is the $x_1$-variable.

\begin{defi}
{\rm 
\label{defsolns}
\noindent
We say that $v$ is a {\it layer solution} of \eqref{problem} if
$v$ is a solution of \eqref{problem} satisfying 
$$
v_{x_1}>0  \quad \mbox{in } \R^{n}\quad\text{ and}
$$
\begin{equation}\label{limlayer}
\lim_{x_1 \rightarrow \pm \infty} v(x)=\pm 1 \quad \text{for every } 
(x_2,\ldots,x_n)\in\RR^{n-1}.
\end{equation}

We say that $u$ is a {\it layer solution} of \eqref{extAlpha}
if it is a solution of \eqref{extAlpha},
\begin{equation}
\label{increasing}
u_{x_1}>0\quad \mbox{on } \partial\R^{n+1}_+, \quad\text{ and}
\end{equation}
\begin{equation}
\label{limits}
\lim_{x_1\to\pm\infty}u(x,0)=\pm 1 \quad\mbox{for every }
(x_2,\ldots ,x_{n})\in \R^{n-1}.
\end{equation}
}
\end{defi}

It is important to emphasize that, for $n\geq 2$, the limits in \eqref{limlayer} and \eqref{limits}
are taken for $(x_2,\ldots ,x_{n})$ fixed, and are not assumed to be
uniform in $(x_2,\ldots ,x_{n})\in\R^{n-1}$.

We will also study solutions $v$ of \eqref{problem} 
which are radially symmetric (not necessarily decreasing) and such that
\begin{equation}
\label{limits2}
\lim_{|x|\to\infty}v(|x|)=0.
\end{equation}

We can now state our main results.
The next theorem provides a necessary condition ---\eqref{nec1} and \eqref{nec2}--- 
on the nonlinearity $f$ to admit a layer solution in $\R$.
In our subsequent paper~\cite{CS2}, this necessary condition will be proven to be also
sufficient for the existence of a layer solution.
It is interesting to point out that conditions \eqref{nec1} and \eqref{nec2} are
independent of the fraction $s\in (0,1)$, and that they are also the necessary
and sufficient conditions for the existence of a layer solution to
the local equation $-v''=f(v)$ in all of $\RR$. 

The theorem also states that families of layer solutions indexed by $s \in (0,1)$ converge as $s$ 
goes to $1$ to a layer solution of the equation $-v''=f(v)$ in $\RR$. 

\begin{theorem}\label{necLayers}
{\rm (i)} Let $s \in (0,1)$ and $f$ any $C^{1,\gamma}(\RR)$ function, 
for some $\gamma >\max(0,1-2s)$. Assume that there exists a layer solution $v$ of  
\begin{equation}\label{eqcc}
(-\partial_{xx})^s v= f(v)\,\,\,\,\mbox{in }\RR ,
\end{equation}
that is, $v$ is a solution of \eqref{eqcc} satisfying 
$$v'>0 \quad \text{in }\RR \qquad\text{and}\qquad 
\lim_{x \rightarrow \pm \infty} v(x)=\pm 1. $$
Then, we have  
\begin{equation}\label{nec1}
G'(1)=G'(-1)=0 
\end{equation}
and 
\begin{equation}\label{nec2}
G>G(1)=G(-1)\,\,\,\,\mbox{in }(-1,1).
\end{equation}

{\rm (ii)} Let $f$ be any $C^{1,\gamma}(\RR)$ function with $\gamma \in (0,1)$.  
Assume that $\left \{ v_s \right \}$, with $s=s_k \in(0,1)$ and $s_k \uparrow 1$, 
is a sequence of layer solutions of 
$$(-\partial_{xx})^s v_s =f(v_s)\,\,\,\mbox{in } \RR,$$
such that $v_s(0)=0.$ Then, there exits a function $\overline v$ such that 
$$\lim_{s=s_k \uparrow 1} v_s=\overline v$$
in the uniform $C^2$ convergence on every compact set of $\RR$. 
Furthermore, the function $\overline v$ is the layer solution of
$$ -\overline v ''=f(\overline v)\,\,\,\mbox{in } \RR$$
with $\overline v(0)=0.$
\end{theorem}

Conditions \eqref{nec1} and \eqref{nec2} express that is $G$ is of double-well type 
and $f$ is of bistable balanced type. Note that 
the statement $G(1)=G(-1)$ is equivalent to 
$$
\int_{-1} ^{1} f =0.
$$

Theorem \ref{necLayers} is actually a consequence of the following Hamiltonian
equality and estimate, which are of independent interest.
We have introduced the Hamiltonian associated to problem \eqref{extAlpha} 
in subsection~1.1 above. The following Hamiltonian identity for layer solutions
in $\R$ states the conservation of the Hamiltonian in ``time'' ---recall that
$x$ plays the role of time variable. Instead, the Hamiltonian inequality below
is the analogue in dimension $1$ of the classical Modica estimate for bounded solutions
of semilinear
equations $\Delta v-G'(v)=0$ in $\mathbb{R}^n$, which states that the kinetic energy 
is bounded at every point by the potential energy, i.e., $(1/2)|\nabla v|^2\leq G(v)$
everywhere in $\mathbb{R}^n$, whenever $G\geq 0$ in $\RR$. 

Notice that our Modica-type estimate is stated for $n=1$. 
It is still an open problem for $n \geq 2$.

The theorem also provides an asymptotic result as $s$ goes to $1$, in which we
recover the classical Hamiltonian identity.

\begin{theorem}\label{modthm}
{\rm (i)} Let $a \in (-1,1)$ and $f$ any $C^{1,\gamma}(\RR)$ function, 
for some $\gamma >\max(0,a)$. Let $n=1$ and $u$ be a layer solution of \eqref{extAlpha}. 
Then, for
every $x\in\RR$ we have 
$\int_0^{+\infty}t^{a}|\nabla u(x,t)|^2 dt<\infty$ and the Hamiltonian equality
\begin{equation}\label{mmequal}
(1+a)\int_0^{+\infty} \frac{1}{2} t^a \left\{u_x^2(x,t)-u_y^2(x,t)\right\} dt=G(u(x,0))-G(1). 
\end{equation}
Furthermore, for all $y \geq 0$ and  $x \in \RR$ we have
\begin{equation}\label{mmestimate}
(1+a) \int_0^y\dfrac{t^{a}}{2}\left\{
u_x^2(x,t)-u_y^2(x,t)\right\} dt < G(u(x,0))-G(1).
\end{equation}

{\rm (ii)} Let $f$ be any $C^{1,\gamma}(\RR)$ function with $\gamma \in (0,1)$, $n=1$ and 
$\left \{ u_a \right \}$, with $a=a_k \in(-1,1)$ and $a_k\downarrow -1$ 
be a sequence of layer solutions of 
\eqref{extAlpha} (with $u$ replaced by $u_a$ for each $a$)
such that $u_a(0,0)=0$. Then, $\lim_{a=a_k \downarrow -1} u_a(\cdot,0)=\overline v$
in the uniform $C^2$ convergence on every compact set of $\R$,
where $\overline v$ is the layer solution of
$ -\overline v ''=f(\overline v)$ in $\RR$ with $\overline v(0)=0$. In addition,
for every $x\in\RR$ we have
$$
\lim_{a\downarrow -1} (1+a)\int_0^{+\infty} \frac{1}{2} t^a (u_a)_x^2(x,t) dt=
 \frac{1}{2} \overline v'(x)^2=G(\overline v(x))-G(1) 
$$
and
$$
\lim_{a\downarrow -1} (1+a)\int_0^{+\infty} \frac{1}{2} t^a (u_a)_y^2(x,t) dt=0.
$$\end{theorem}

We emphasize once again that the previous estimate \eqref{mmestimate} is pointwise and nonlocal. 

The asymptotic result when $a \rightarrow -1$ (i.e., $s\to 1$) of part (ii) 
in the previous theorem allows to recover 
from \eqref{mmequal} the standard conservation of the Hamiltonian for the Laplacian. 
This will be presented in section~6 below.

In the case of radial solutions with limit at infinity, 
we establish the following result.
Here the dimension $n$ is arbitrary. 

\begin{theorem}\label{necBounds}
Let $s \in (0,1)$ and $f$ any $C^{1,\gamma}(\RR)$ function, 
for some $\gamma >\max(0,1-2s)$. 

Let $n>1$ and $v=v(x)=v(|x|)$ be a nonconstant radial solution of  
\begin{equation*}
(-\Delta )^s v= f(v)\,\,\,\,\mbox{in }\RR^n 
\end{equation*}
satisfying 
$$\lim_{|x| \rightarrow +\infty} v(|x|)=0. $$

Then, we have  
$$f(0)=0=G'(0) \,\,\,\,\mbox{and}\,\,\,\,\,G(0)>G(v(0)).$$
If, in addition, $v$ is decreasing in $|x|$, then 
$$f'(0)=-G''(0) \leq 0. $$ 
\end{theorem}

The statement $G(0)>G(v(0))$ is equivalent to 
$$
\int_0 ^{v(0)} f >0.
$$

As in the case of layer solutions, Theorem \ref{necBounds} relies on 
the following statement about the Hamiltonian. 
In this next theorem we do not assume $u(\cdot,0)$ to have a limit at infinity.
 
\begin{theorem}\label{modthmGS}
Let $a \in (-1,1)$ and $f$ any $C^{1,\gamma}(\RR)$ function, 
for some $\gamma >\max(0,a)$. 

Let $n \geq 1$ and $u$ be a bounded
solution of \eqref{extAlpha}
which is radial in $x$, i.e., $u(x,y)=u(|x|,y)$.

Then, $\int_0^{+\infty}t^a|\nabla u(r,t)|^2 dt<\infty$
for every $r\geq 0$, and the quantity 
\begin{equation}\label{mmequalGS}
(1+a)\int_0^{+\infty} \frac{t^a}{2} \Big \{ u_r^2(r,t) - u_ y^2(r,t) \Big \}\,dt -G(u(r,0))
\end{equation}
is nonincreasing in $r\geq 0$. 
\end{theorem}

\section{Local realization of the fractional Laplacian and results on
  degenerate elliptic equations}

This section is concerned with the relation between the local problem
\eqref{extAlpha} and the nonlocal problem \eqref{problem}. 
We collect also several results on degenerate elliptic equations
with $A_2$ weights. 

We first introduce the spaces
$$H^s(\RR^n)=\left \{ v \in L^2(\RR^n)\,\,:\,\,|\xi|^{s} (\mathcal F v)(\xi) \in L^2(\RR^n) 
\right \},$$
where $s \in (0,1)$ and $\mathcal F$ denotes Fourier transform.
For $\Omega \subset \RR^{n+1}_+$ a Lipschitz domain (bounded or unbounded) and $a \in (-1,1)$,
we denote
$$H^1(\Omega,y^a)=\left \{ u \in L^2
(\Omega,y^a \,dx\,dy)\,\,:\,\,|\nabla u| \in L^2(\Omega,y^a \,dx\,dy) \right \}.$$

\subsection{Local realization of the fractional Laplacian}
The fractional  
Laplacian can be defined in various ways, which we review now. 
It can be defined using Fourier transform by 
$$
\mathcal F((-\Delta)^s v) = \left| \xi \right|^{2s}\mathcal F(v), 
$$
for $v\in H^s(\R^n)$. It can also
be defined through the kernel representation (see the book by Landkof \cite{landkof})
\begin{equation}\label{defPV}
(-\Delta)^{s} v(x)=C_{n,s}\ \textrm{P.V.}\int_{\mathbb{R}^n} \frac{v(x)-v(\overline x)}
{|x-\overline x|^{n+2s}}\,d\overline x,
\end{equation}
for instance for $v\in\mathcal{S}(\RR^n)$, the Schwartz space of rapidly decaying functions. 
One can also define the fractional Laplacian acting on spaces of functions with weaker regularity. 
Indeed, following \cite{SilCPAM}, one defines the space  $\mathcal S_s(\RR^n)$
of $C^\infty$ functions $v$ such that for every $k \geq 0$, the quantity $(1+|x|^{n+2s})D^kv(x)$ 
is bounded. We denote $\mathcal S'_s(\RR^n)$ its dual. It is then possible to check that $(-\Delta)^s$ 
maps $\mathcal S (\RR^n)$ into  $\mathcal S_s (\RR^n)$. By duality, this allows to define the 
fractional Laplacian for functions in the space
$$\mathcal{L}_s(\RR^{n}) := \left \{ v \in L^1_{\rm loc}(\RR^n) \, :\,  \int_{\RR^{n}} 
\frac{|v(x)|}{(1+|x|)^{n+2s}}\, dx< \infty \right \}$$
$$=L^1_{\rm loc}(\RR^n) \cap \mathcal S'_s(\RR^n).$$
For $v\in \mathcal{L}_s(\RR^{n}) \cap C^2_{\rm loc}(\RR^{n})$, the integral in \eqref{defPV} is well 
defined. This is clear for
$|\overline x|$ large. For the Cauchy principal value to be well defined
(as $\overline x\to x$), it suffices to assume that $v$ is $C^2_{\rm loc}(\RR^{n})$. 
In particular, expression \eqref{defPV} defines the operator on the type of
solutions that we consider, since they will always be bounded in $\RR^n$ and locally
$C^2$. See \cite{SilCPAM, landkof} for more comments on the various definitions
of the fractional Laplacian and their agreement.
We refer the reader to the book by Landkof \cite{landkof} where an
extensive study of integro-differential operators with Martin-Riesz
kernels, i.e., kernels of the type (up to a normalizing constant) $1/|z|^{n+2s}$ is presented.

It is well known that one can see the operator
$(-\Delta)^{1/2}$ by considering it as the Dirichlet to Neumann operator associated to the
harmonic extension in the halfspace, paying the price to add a new
variable. In \cite{cafS}, Caffarelli and Silvestre proved that such a kind of
realization is also possible for any power of the Laplacian between $0$
and $1$, as follows.

Given $s\in(0,1)$, let $a=1-2s\in(-1,1)$. It is well known that the space $H^s(\R^n)$ 
coincides with the trace on $\partial\RR^{n+1}_{+}$ of $H^1(\RR^{n+1}_+,y^a)$.
In particular, every $v\in H^s(\R^n)$ is the trace of a function $u\in L^2_{\rm loc}(\RR^{n+1}_+,y^a)$ 
such that $\nabla u \in L^2(\RR^{n+1}_+,y^a)$.  
In addition, the function $u$ which minimizes
\begin{equation} \label{argmin} 
{\rm min}\left\{ \int_{\R^{n+1}_{+}}y^a \left| \nabla u \right|^2\;dx dy \; : \;  
u|_{\partial\R^{n+1}_{+}}=v\right\}
\end{equation} 
solves the Dirichlet problem
\begin{equation}\label{bdyFrac2} 
\left \{
\begin{aligned} 
L_a u:= \textrm{div\,} (y^a \nabla u)&=0 \qquad 
{\mbox{ in $\RR^{n+1}_+$}} 
\\
u&= v  
\qquad{\mbox{ on $\partial\RR^{n+1}_+$.}}\end{aligned}\right . 
\end{equation} 
By standard elliptic regularity, $u$ is smooth in $\R^{n+1}_{+}$. It turns out that 
$-y^a u_{y} (\cdot,y)$ converges in $H^{-s}(\RR^n)$ to a distribution 
$h\in H^{-s}(\RR^n)$ as $y\downarrow 0$. That is, $u$ weakly solves
\begin{equation}\label{bdyFrac3} 
\left \{
\begin{aligned} 
\textrm{div\,} (y^a \nabla u)&=0 \qquad 
{\mbox{ in $\RR^{n+1}_+ $}} 
\\
-y^a \partial_y u &= h
\qquad{\mbox{ on $\partial\RR^{n+1}_+$.}}\end{aligned}\right . 
\end{equation}
Consider the Dirichlet to Neumann operator 
$$
\begin{aligned}
&\Gamma_a: H^s(\RR^n)\to H^{-s}(\RR^n)\\
&\qquad \quad v\mapsto \Gamma_{a}(v)=  h:=
 - \displaystyle{\lim_{y \rightarrow 0^+}} y^{a} \partial_y u, 
\end{aligned}
$$
where $u$ is the solution of \eqref{bdyFrac2}. 
Then, we have:

\begin{theorem}[\cite{cafS}]\label{realization}
For every $v\in H^s(\R^n)$,
\begin{equation*}
(-\Delta)^s v= d_s\Gamma_{a}(v)= 
- d_s \displaystyle{\lim_{y \rightarrow 0^+}} y^{a} \partial_y u,
\end{equation*}
where $a=1-2s$, $d_s$ is a positive constant depending only on~$s$,
and the equality holds in the distributional sense.
\end{theorem}  

In other words, given $h\in H^{-s}(\RR^n)$, a function $v\in H^s(\RR^n)$ solves the equation 
$(-\Delta)^{s} v=d_sh$ in $\mathbb R^n$
if and only if its extension $u\in H^1(\RR^{n+1}_+,y^a)$ solves \eqref{bdyFrac3}.
By duality, the same relation can be stated when $v \in 
\mathcal{L}_s(\RR^{n})$ ---as it is the case of the solutions considered in this paper.

\subsection{Degenerate elliptic equations with $A_2$ weights}
According to the previous result, we must study the
operator $L_{a}=\textrm{div\,}(y^{a} \nabla)$ in $\mathbb{R}^{n+1}_+$,
whose weight $y^{a}$ belongs to the class $A_2$ since $a\in (-1,1)$. 
Through a reflection method, it will be useful to consider the equation 
in domains $\Omega\subset\RR^{n+1}$ not necessarily contained in 
$\RR^{n+1}_+$. In such case, we extend the weight $y^a$ by $|y|^a$
for $y<0$. That is, we define
\begin{equation}\label{operallspace}
L_{a}u:=\textrm{div\,}(|y|^{a} \nabla u) \qquad\text{in }
\Omega\subset\mathbb{R}^{n+1}.
\end{equation}

In a series of papers (\cite{FKS,FJK}),  
Fabes, Jerison, Kenig, and Serapioni developed a systematic theory for this class
of operators: existence of weak solutions, Sobolev embeddings, Poincar\'e inequality, Harnack
inequality, local solvability in H\"older spaces, and estimates on the
Green's function. 

In particular, as a consequence of a Poincar\'e inequality related to  $A_2$
weights, they established an existence result (via the Lax-Milgram theorem)
and a H\"older continuity result.
The following three results for $L_a$ as in \eqref{operallspace}, with
$a\in (-1,1)$ follow from results of \cite{FKS}, stated there for general
$A_2$ weights. More precisely, they follow respectively from 
Theorem~2.2, Theorems~2.3.12 and ~2.3.15 (and Remark~1 preceding it),
and Lemma~2.3.5 of \cite{FKS}.

\begin{theorem} [Solvability in Sobolev spaces \cite{FKS}] \label{solveFKS}
Let $\Omega\subset\RR^{n+1}$ be a smooth bounded domain, 
$h=(h_1,...,h_{n+1})$ satisfy $|h|/|y|^a \in L^2(\Omega,|y|^a)$, and 
$g \in H^1(\Omega,|y|^a)$. Then, there exists a unique solution  
$u\in H^1(\Omega,|y|^a)$ of 
$L_a u=-{\rm div }\, h$ in $\Omega$ with $u-g \in H^1_0(\Omega,|y|^a)$.
\end{theorem}

\begin{theorem} [H\"older local regularity \cite{FKS}]\label{HolderFKS}
Let $\Omega\subset\RR^{n+1}$ be a smooth bounded domain
and $u$ a solution of $L_a u=-{\rm div\,} h$ in 
$\Omega$, where $|h|/|y|^a \in L^{2(n+1)}(\Omega,|y|^a)$. 
Then, $u$ is H\"older continuous in $\Omega$ with a H\"older exponent depending only on $n$ and $a$.
\end{theorem}

\begin{theorem} [Harnack inequality \cite{FKS}] \label{HarnackFKS}
Let $u$ be a positive solution of $L_a u=0$ in 
$B_{4R}(x_0)\subset\RR^{n+1}$. Then, $\sup_{B_R(x_0)} u \leq C \inf_{B_R(x_0)} u$
for some constant $C$ depending only on $n$ and $a$ ---and in particular, independent of~$R$.

As a consequence, bounded solutions of $L_a u=0$ in all of
$\RR^{n+1}$ are constant.
\end{theorem}

The last statement is proved applying the previous Harnack inequality 
to $u-\inf_{\RR^{n+1}} u$ in $B_{4R}(0)$ and letting $R\to \infty$.

\begin{coro}
\label{uniqueness}
Problems \eqref{bdyFrac2} and \eqref{bdyFrac3} admit at most one
solution $u$ with $u$ bounded and continuous in $\overline{\RR^{n+1}_+}$
---up to an additive constant in the case of \eqref{bdyFrac3}.
\end{coro}

\begin{proof}
The difference $w$ of two solutions would solve the homogeneous problem.
We can then perform odd reflection for problem \eqref{bdyFrac2} and 
even reflection for problem \eqref{bdyFrac3},
to obtain a  bounded solution of $L_a w=0$ in all of
$\RR^{n+1}$. By Theorem~\ref{HarnackFKS}, $w$ is constant, which finishes the proof.

We remark that in case of the Neumann problem \eqref{bdyFrac3},
the previous Liouville and uniqueness results also follow from the Harnack
inequality in $\RR^{n+1}_+$ that we prove in 
Lemma~\ref{harnack} below. 
\end{proof}

The existence of a bounded solution for \eqref{bdyFrac2} and \eqref{bdyFrac3}
is stated below in this same section.

\subsection{A duality principle}
An important property of the operator $L_a$ is the
following duality property. It relates the Neumann problem for the operator
$L_a$ with the Dirichlet problem for $L_{-a}$, the operator with the inverse 
weight.

\begin{propo}[\cite{cafS}]\label{duality}
Assume that $h \in C(\RR^n)$, $u \in C^2(\RR^{n+1}_+)$, and $y^a \partial_y u \in 
C(\overline{\RR^{n+1}_+})$.  If $u$ is a classical solution of 
\begin{equation*}
\begin{cases}
{\rm div\,} (y^{a}\,\nabla u)=0&\text{ in } \mathbb{R}^{n+1}_+\\ 
\displaystyle{\frac{\partial u}{\partial{\nu}^{a}}}
=h &\text{ on } \partial\mathbb{R}^{n+1}_+ ,
\end{cases}
\end{equation*} 
then $w=-y^a \partial_y u$ is a classical solution of 
\begin{equation*}
\begin{cases}
{\rm div\,} (y^{-a}\,\nabla w)=0&\text{ in } \mathbb{R}^{n+1}_+\\ 
w=h &\text{ on } \partial\mathbb{R}^{n+1}_+ .
\end{cases}
\end{equation*}
\end{propo}

The previous duality property is related, in dimension two, to some generalized
Cauchy-Riemann conditions that we describe. Indeed, writing $L_a u=0$ in $\mathbb{R}_+^2$ as 
\begin{equation*}
\partial_x (y^a \partial_x u)+\partial_y (y^a \partial_y u)=0,
\end{equation*}  
we see that the associated conjugate function $\tilde u$ is such that 
\begin{equation*}
\begin{cases}
y^a \partial_y u= \partial_x \tilde u,\\
-y^a \partial_x u= \partial_y \tilde u,
\end{cases}
\end{equation*} 
hence satisfying generalized Cauchy-Riemann conditions. 
The function $\tilde u$ is the $a$-conjugate of $u$. Similarly, $u$ is the $-a$ conjugate of 
$\tilde u.$ Complexifying the problem by denoting $\varphi =u+i\tilde u$, it is easy to see 
that $\varphi$ satisfies 
\begin{equation}\label{beltrami}
\overline{\partial} \varphi =\nu(y) \overline{\partial \varphi}
\end{equation}
where $\overline{\partial}=\partial_x +i \partial_y$ is the standard
$\overline{\partial}$-operator and 
\begin{equation*}
\nu(y)=\frac{1+y^a}{1-y^a}. 
\end{equation*}
Equation \eqref{beltrami} is called conjugate Beltrami equation and has been extensively studied 
in the Calderon problem (see \cite{astala} and references therein). 

\subsection{Fundamental solutions}
Concerning the operator $L_a$ involved in the extension or Dirichlet problem for the
fractional Laplacian, one has the following represention
formula through a Poisson kernel.

\begin{propo} [\cite{cafS}] \label{poisson}
Given $s \in (0,1)$, let $a=1-2s \in (-1,1)$. The function 
$$P_s(x,y)=p_{n,s}\frac{y^{2s}}{\Big ( |x|^2+y^2\Big)^\frac{n+2s}{2}}
=p_{n,s}\frac{y^{1-a}}{\Big ( |x|^2+y^2\Big)^\frac{n+1-a}{2}}$$
is a solution of 
\begin{equation}\label{POI}
\begin{cases}
{\rm div\,} \, (y^{a}\,\nabla P_s)=0&\text{ in } \mathbb{R}^{n+1}_+\\ 
P_s= \delta_0&\text{ on } \partial\mathbb{R}^{n+1}_+=\RR^n ,
\end{cases}
\end{equation} 
where $\delta_0$ is the delta distribution at the origin, and
$p_{n,s}$ is a positive constant depending only on $n$ and~$s$
chosen such that, for all $y>0$, 
$$\int_{\RR^{n}} P_s(x,y)\,dx=1. $$
\end{propo}

\begin{rema}\label{equipb}
{\rm
As a consequence of Corollary~\ref{uniqueness} above,
we have that for $v\in (C \cap L^\infty) (\RR^n)$,
the convolution in the $x$-variables
$$
u(\cdot,y)=P_s(\cdot, y)  * v
$$
is the unique solution of \eqref{bdyFrac2} in $(C \cap L^\infty) (\overline{\RR^{n+1}_+})$. 

As a consequence,
for $v$ a bounded $C^2_\textrm{loc}(\RR^n)$ function, $v$ 
is a solution of \eqref{problem} if and only if 
$$u(\cdot,y)=P_s(\cdot,y) * v$$
is a solution of \eqref{extAlpha} (with $f$ replaced by $(1+a) d_s^{-1} f
= 2(1-s)d_s^{-1} f$) whose trace on $\partial \RR^{n+1}_+$ is $v$.
Recall that $d_s$ is the constant from \eqref{cttNeumann}. 
}
\end{rema}

{From} the previous duality principle (Proposition~\ref{duality})
and the knowledge of the 
fundamental solution (the Poisson kernel) of the Dirichlet
problem \eqref{bdyFrac2}, we can find the fundamental solution
of the fractional Laplacian or, equivalently, the fundamental solution  
of the Neumann problem \eqref{bdyFrac3}.

\begin{propo} [\cite{cafS}] \label{fundNeum}
Given $s \in (0,1)$, let $a=1-2s \in (-1,1)$. The function 
$$
\Gamma_s(x,y)=e_{n,s} |(x,y)|^{2s-n}=e_{n,s} |(x,y)|^{1-a-n}
$$
is a solution of 
\begin{equation}\label{PON}
\begin{cases}
{\rm div\,} \, (y^{a}\,\nabla \Gamma_s)=0&\text{ in } \mathbb{R}^{n+1}_+\\ 
-y^a\partial_y \Gamma_s= \delta_0&\text{ on } \partial\mathbb{R}^{n+1}_+=\RR^n ,
\end{cases}
\end{equation} 
where $\delta_0$ is the delta distribution at the origin, and
$e_{n,s}$ is a positive constant depending only on $n$ and~$s$.
\end{propo} 

As a consequence, $\Gamma_s(x,0)=e_{n,s} |x|^{2s-n}$ is, up to a multiplicative constant,
the fundamental solution of $(-\Delta)^s$ in $\RR^n$.

\begin{rema}\label{existNeum}
{\rm
As a consequence of Corollary~\ref{uniqueness} above,
we have that for  $h\in C_c  (\RR^n)$ ($h$ continuous with compact support),
the convolution in the $x$-variables
$$
u(\cdot,y)=\Gamma_s(\cdot, y)  * h
$$
is the unique (up to an additive constant) 
solution of \eqref{bdyFrac3}  in $(C \cap L^\infty) (\overline{\RR^{n+1}_+})$.
Thus, its trace
$$
v= |x|^{2s-n} * h
$$
is up to a multiplicative
constant, the unique (up to an additive constant) continuous and bounded solution of 
$$(-\Delta)^s v=h \qquad \text{in }\RR^n.$$
}
\end{rema}

\begin{rema}\label{constant}
{\rm
The normalizing constant $C_{n,s}$ in \eqref{fracpv} is given by 
$$C_{n,s}=\pi^{-\frac{n}{2}}2^{2s}\frac{\Gamma(\frac{n+2s}{2})}{-\Gamma(-s)}=\pi^{-\frac{n}{2}}2^{2s}
\frac{\Gamma(\frac{n+2s}{2})}{\Gamma(1-s)}s=
\pi^{-\frac{n}{2}}2^{2s}\frac{\Gamma(\frac{n+2s}{2})}{\Gamma(2-s)}s(1-s).$$
In particular, up to positive multiplicative constants, 
$C_{n,s}$ behaves as $s$ and $1-s$ for $s \downarrow 0$ and $s \uparrow 1$, respectively.  

Let us also make some comments on the constant $d_s$ in the
Caffarelli-Silvestre extension problem ---see Theorem~\ref{realization}. Its value is given by 
$$d_s=2^{2s-1}\frac{\Gamma(s)}{\Gamma(1-s)}.$$
The fact that this constant does not depend on $n$ is already proved in section~3.2 of \cite{cafS}. 
Its precise value appears in several papers; see e.g. \cite{stinga,frank}. Using that 
$s\Gamma (s)=\Gamma(s+1)$ and $(1-s)\Gamma (1-s)=\Gamma(2-s)$, we deduce, respectively, that 
\begin{equation}\label{limctts}
\frac{d_s}{(2s)^{-1}} \rightarrow 1 \,\,\,\text{as } s\downarrow 0 \,\,\,\,\,\,\,\mbox{and}\,\,\,\, 
\frac{d_s}{2(1-s)} \rightarrow 1 \,\,\, \text{as } s\uparrow 1.
\end{equation}

When $u$ solves problem \eqref{extAlpha}, its boundary condition
$(1+a)\partial_{\nu_a}u= 2(1-s) \partial_{\nu_a}u =f(u)$ gives that the trace
of $u$ solves
$$
(-\Delta)^s \{u(,\cdot,0)\} = \frac{d_s}{2(1-s)} f(u(,\cdot,0)).
$$
}
\end{rema}

\section{Preliminary results on elliptic problems involving
$L_a$: Schauder estimates, maximum principles, and a Liouville theorem}

This section is devoted to the proof of several general results
concerning problem \eqref{extAlpha}. The following definition provides the concept of 
weak solution for \eqref{extAlpha}. More generally, we consider the problem 
\begin{equation}\label{temp}
\begin{cases}
\textrm{div\,} (y^a \nabla u)=0&\text{in $B_R^+$}\\
-y^a u_y=h&\text{on $\Gamma_R^0$,}
\end{cases} 
\end{equation}
where we have used the notation introduced in the
beginning of section~2.

\begin{defi}
{\rm
Given $R>0$ and a function $h \in L^1(\Gamma^0_R)$, we say that $u$ is a 
weak solution of \eqref{temp} if 
$$y^a |\nabla u|^2 \in L^1(B_R^+)$$
and 
\begin{equation}
\int_{B_R^+} y^{a} \nabla u \cdot \nabla \xi -\int_{\Gamma^0_R} h
\xi =0
\end{equation}
for all $\xi \in C^1(\overline{B_R^+})$ such that $\xi\equiv 0$ on $\Gamma^+_R$.}
\end{defi}

\begin{rema}\label{MPweak}
{\rm
The (weak) maximum principle holds for weak solutions of \eqref{temp}.
More generally, if $u$ weakly solves
\begin{equation}\label{pbMPweak}
\begin{cases}
-\textrm{div\,} (y^a \nabla u)\geq 0&\text{in $B_R^+$}\\
-y^a u_y\geq 0 &\text{on $\Gamma_R^0$}\\
u\geq 0 &\text{on $\Gamma_R^+$,}
\end{cases}
\end{equation}
then $u\geq 0$ in $B_R^+$. This is proved simply multiplying the weak
formulation by the negative part $u^-$ of $u$. 

In addition, one has
the strong maximum principle: either $u\equiv 0$ or 
$u>0$ in $ B_R^+\cup \Gamma_R^0$. That $u$ cannot vanish at an 
interior point follows from the classical strong maximum principle
for strictly elliptic operators. That $u$ cannot vanish at a
point in $\Gamma_R^0$ follows from the Hopf principle 
that we establish below (see Proposition~\ref{hopf}) or
by the strong maximum principle of \cite{FKS}.

Note that the same weak and strong maximum principles (and proofs) hold
in other bounded domains of $\RR^{n+1}_+$ different than $B_R^+$. It also holds for 
the Dirichlet problem in $B_R^+$, i.e., replacing
the Neumann condition in \eqref{pbMPweak} by $u\geq 0$ on $\Gamma_R^0$.
}
\end{rema} 

\subsection{Schauder estimates}
In the following, we prove several estimates for solutions
of \eqref{problem} and \eqref{extAlpha}, as well as for solutions of the Neumann problem  
\eqref{temp} in $B_R^+.$

\begin{rema}
{\rm
Note that the function $u(x,y)=y^{1-a}$ solves $L_a
u=0$. Therefore, one cannot expect the H\"older regularity in $y$ to be higher than 
than $C^{\min(1,1-a)}$ up to the boundary $\left \{ y=0\right \}$. Thus, for $s < 1/2$ 
(i.e. $a >0$), there are solutions of $L_a u=0$ vanishing on $\left \{ y=0\right \}$ which are 
not $C^1$ up to $\left \{ y=0\right \}$. 
}
\end{rema}

The next lemmas provide several regularity results. We start with estimates on the 
nonlocal equation. 

\begin{lem}\label{regNL}
Let $f$ be a $C^{1,\gamma}(\RR)$ function with $\gamma >\max(0,1-2s)$. 
Then, any bounded solution of 
$$(-\Delta)^s v =f(v)\,\,\,\mbox{in }\RR^n$$
is $C^{2,\beta}(\RR^n)$ for some $0 <\beta < 1$ depending only on $s$ and $\gamma$.

Furthermore, given $s_0>1/2$ there exists $0 <\beta < 1$ depending only on $n$,
$s_0$, and $\gamma$ ---and hence independent of $s$--- such that for every $s>s_0$, 
$$\|v\|_{C^{2,\beta}(\RR^n)} \leq C $$
for some constant $C$ depending only on $n$, $s_0$, $\|f\|_{C^{1,\gamma}}$,
and $\|v\|_{L^{\infty}(\RR^n)}$ ---and hence independent of $s \in (s_0,1)$.

In addition, the function defined by $u=P_s \, * \, v$ (where $P_s$ is the Poisson kernel 
in Proposition 3.7) satisfies for every $s>s_0$,
$$\|u\|_{C^{\beta}(\overline{\RR^{n+1}_+})} + \|\nabla_x u\|_{C^{\beta}(\overline{\RR^{n+1}_+})}
+\|D^2_x u\|_{C^{\beta}(\overline{\RR^{n+1}_+})} \leq C $$
for some constant $C$ independent of $s\in (s_0,1)$, indeed 
depending only on the same quantities as the previous one.
\end{lem}

\begin{proof}
Since $v$ is bounded, $f(v)$ is also bounded. Applying Proposition 2.9 in \cite{SilCPAM}, we have 
\begin{itemize}
\item If $s \leq 1/2$, then for any $\alpha < 2s$, $v \in C^{0,\alpha}(\RR^n)$. 
\item If $s >1/2$, then for any $\alpha < 2s-1$, $v \in C^{1,\alpha}(\RR^n)$. 
\end{itemize}
This implies in particular that $f(v)$ is 
$C^\alpha(\RR^n)$. Applying now
Proposition 2.8 in \cite{SilCPAM}, we have 
\begin{itemize}
\item If $\alpha +2s \leq 1$, then $v \in C^{0,\alpha+2s}(\RR^n)$. 
\item If $\alpha +2s >1$, then $v \in C^{1,\alpha+2s-1}(\RR^n)$. 
\end{itemize}
Therefore, iterating the procedure a finite number of times, one gets that $v \in C^{1,\sigma}(\RR^n)$
for some $\sigma \in(0,1)$ depending only on $s$. Indeed, if $\alpha +2s >1$, then one can take 
$\sigma=\alpha+2s-1.$ On the other hand if $\alpha +2s \leq 1$, we have that $f(u)$ is $C^{0,\alpha +2s}$.
As a consequence, one gets that $u$ is $C^{0,\alpha +4s}$. Hence iterating a finite number of times, 
we will end up with $\alpha + k2s >1$ for some integer $k$.  This gives the $C^{1,\sigma}$ regularity. 

We now differentiate the equation to obtain
$$(-\Delta)^s v_{x_i}=f'(v)v_{x_i}\,\,\,\mbox{in }\RR^n$$
for $i=1,...,n$, with  $v_{x_i}$ and $f'(v)$ belonging to  
$C^{0,\sigma}(\RR^n)$ provided we take $\sigma < \gamma$. Therefore, applying Proposition 2.8 of 
\cite{SilCPAM} we obtain 
that $v_{x_i} \in C^{0,\sigma +2s}(\RR^n)$. 
We iterate this procedure a finite number of times (as long as
the H\"older exponent is smaller than $\gamma$). Now, since by assumption
$\gamma +2s >1$, we finally arrive at  $v_{x_i} \in C^{1,\beta}(\RR^n)$,
and thus $v \in C^{2,\beta}(\RR^n)$ for some $\beta >0$ depending only on $s$ and $\gamma$.

For the second point of the lemma, we write the nonlocal equation as 
$$-\Delta v= (-\Delta)^{1-s}f(v)\,\,\,\,\mbox{in}\,\,\RR^n. $$
A careful look at the proof of Proposition 2.5 in \cite{SilCPAM} shows the following. 
Given $s_0 >1/2$, if $0 < 2-2s_0<\alpha <1$, the operator  $(-\Delta)^{1-s}$ maps 
$C^\alpha(\RR^n)$ into $C^{\alpha-2+2s}(\RR^n)$ continuously with a constant independent of $s$. 
Here we use that the constant $C_{n,s}$ in \eqref{fracpv} is uniformly bounded as 
$s \uparrow 1$ ---see Remark \ref{constant}. As a consequence, applying $C^{2,\beta}$ estimates for 
Poisson equation, we deduce that $v \in C^{2,\alpha-2+2s_0}(\RR^n)=C^{2,\beta}(\RR^n)$ and a 
$C^{2,\beta}$ estimate with a constant independent of $s$ ---indeed depending on the quantities 
in the statement of the lemma.

We now come to the last point of the lemma. 
Let $v \in ( L^\infty \cap C^\beta) (\RR^n)$ for some $\beta \in (0,\min(1,2s_0)]$ ---here we allow 
$s_0 \in (0,1)$. 
We claim that there exists a constant $C$ depending on $n$ and $s_0$, independent of $s>s_0$, such 
that the function $u$ defined by 
$$u(\cdot,y)=P_s(\cdot,y) \,* \, v$$
is $( L^\infty \cap C^\beta) (\RR^{n+1}_+)$ with the estimate 
$$\|u\|_{C^\beta(\overline{\RR^{n+1}_+})} \leq C \|v\|_{C^\beta(\RR^n)}. $$

Applying this fact to $v$, $v_{x_i}$ and $v_{x_ix_j}$, we conclude the statement of the lemma. 
To prove the claim, the Poisson kernel $P_s$ writes 
$$P_s(x,y)=\frac{1}{y^n} H_s(\frac{x}{y})$$
where 
$$H_s(\xi)=p_{n,s} \frac{1}{(1+|\xi|^2)^{\frac{n+2s}{2}}}.$$
The constant $p_{n,s}$ is such that $\int_{\RR^n} H_s(\xi)\,d\xi=1$ and  
therefore it is bounded uniformly in $s$, for $s >s_0$. We have 
$$u(x,y)=\int_{\RR^n} v(x-y \xi)H_s(\xi) \,d\xi $$
and then 
$$|u(x_1,y_1)-u(x_2,y_2)| =\left |\int_{\RR^n} \Big \{ v(x_1-y_1 \xi)-v(x_2-y_2 \xi) \Big \} 
H_s(\xi) \,d\xi \right |$$
$$\leq C \Big \{ |x_1-x_2 |^\beta +|y_1-y_2 |^\beta \int_{\RR^n} H_s(\xi) 
|\xi |^\beta d\xi\Big \}\|v\|_{C^\beta(\RR^n)}.$$
Thus  we deduce  the desired result taking $\beta <2s_0 <2s. $
\end{proof}

The next lemma provides estimates for solutions of the
Neumann problem in a half-ball.
 
\begin{lem}
\label{regularity1} 
Let $a \in (-1,1)$ and $R>0$. Let $\varphi \in C^\sigma (\Gamma^0_{2R})$ for some $\sigma \in (0,1)$ and 
$u \in L^\infty(B^+_{2R}) \cap H^1(B^+_{2R},y^a)$ be a weak solution of
\begin{equation*}
\label{problemBR}
\begin{cases}
L_a u=0&\text{ in } B^+_{2R}\subset\RR^{n+1}_+\\ 
\displaystyle{\frac{\partial u}{\partial\nu^a}}
=\varphi&\text{ on } \Gamma^0_{2R}.
\end{cases}
\end{equation*}

Then, there exists  $\beta \in (0,1)$ depending only on $n$, $a$, and $\sigma$, 
such that $u \in C^{0,\beta}(\overline{B_R^+})$ and 
$y^a u_y \in C^{0,\beta}(\overline{B_R^+})$. 

Furthermore, there exist constants $C^1_R$ and $C^2_R$ depending only on $n$, $a$, 
$R$, $\|u\|_{L^\infty(B_{2R}^+)}$ and also on $\|\varphi \|_{L^\infty(\Gamma^0_{2R})}$ (for $C^1_R$) 
and $\|\varphi \|_{C^\sigma(\Gamma^0_{2R})}$ (for $C^2_R)$ , such that 
$$\|u\|_{C^{0,\beta}(\overline{B_R^+})} \leq C^1_R$$
and 

$$\|y^a u_y\|_{C^{0,\beta}(\overline{B_R^+})} \leq C^2_R.$$
\end{lem}

\begin{proof}
Multiply $\varphi$ by a cut-off function identically $1$ in $\Gamma^0_{3R/2}$ and call it 
$\overline \varphi$. Thus $\overline \varphi$ is $C^\sigma_c(\RR^n)$ and 
$\overline \varphi \equiv \varphi$ in $\Gamma^0_{3R/2}$.  
Let $\overline u$ be the solution of 
\begin{equation*}
\begin{cases}
L_a \overline{ u}=0&\text{ in }\R^{n+1}_+\\ 
\displaystyle{\frac{\partial \overline{ u}}{\partial\nu^a}}
=\overline \varphi&\text{ on } \partial \R^{n+1}_+
\end{cases}
\end{equation*}
given in Remark~\ref{existNeum}. We have that $\overline u$
is continuous and bounded.

Then, the function $\overline{ u}|_{\partial \RR^{n+1}_+}$ solves 
$$(-\Delta)^s \overline{u} =d_s\overline \varphi\,\,\,\,\mbox{in}\,\,\RR^n $$
and by Proposition 2.9 in \cite{SilCPAM}, we have that $\overline{ u} \in C^{0,\beta}(\RR^n)$ 
for some $\beta \in (0,1)$ depending only on $n$ and $s=\frac{1-a}{2}$.  

We consider now $w=u-\overline{u}$. Then, $w$ is a bounded 
function and a weak solution of  
\begin{equation*}
\begin{cases}
L_a w=0&\text{ in } B^+_{3R/2}\subset\R^{n+1}_+\\ 
\displaystyle{\frac{\partial w}{\partial\nu^a}}
=0&\text{ on } \Gamma^0_{3R/2}.
\end{cases}
\end{equation*}
Reflecting evenly the function $w$ with respect to $\left \{ y=0 \right \}$, 
the reflected function $\tilde w$ satisfies in the weak sense the problem 
\begin{equation*}
\textrm{div\,} ( |y|^a \nabla \tilde w)=0\,\,\,\text{ in } B_{3R/2}\subset \R^{n+1}.
\end{equation*}
Since the weight $|y|^a$ is $A_2$ and the function $\tilde w$ is bounded, 
the regularity theory in \cite{FKS} (see Theorem \ref{HolderFKS} above) ensures that $\tilde w$ 
is $C^{0,\beta}$ 
for some $\beta \in (0,1)$ depending only on $n$ and $s$.

Putting these two results together ensures that $u \in C^{0,\beta}(\overline{B_R^+})$ 
for some $\beta \in (0,1)$ depending on $n$ and $s$. Furthermore, we get the estimate 
$$\|u\|_{C^{0,\beta}(\overline{B^+_R})} \leq C^1_R $$
for some constant $C^1_R$ as in the statement of the lemma (depending on $\|\varphi\|_{L^\infty}$ 
and not on $\|\varphi\|_{C^\sigma}$).   

On the other hand, by Proposition \ref{duality}, $\psi(x,y):=-y^a\,u_y(x,y)$ satisfies 
\begin{equation*}
\begin{cases}
L_{-a} \psi=0&\text{ in }\,B^+_{2R}\\ 
\psi= \varphi&\text{ on }  \Gamma^0_{2R}.
\end{cases}
\end{equation*}
We introduce
$$\tilde \psi(x,y)=P_{\bar s}(\cdot,y) \, * \, \overline \varphi,  $$
where $\bar s$ is such that $1-2\bar s =-a$.  Recall that $\overline \varphi \in C^\sigma (\RR^n)$ 
has compact support. Thus, $\tilde \psi$ is bounded and $C^\beta(\overline{\RR^{n+1}_+})$ if 
$\beta \leq \min(\sigma, 2\overline s )$ (for this, recall the argument on convolutions at 
the end of the proof of Lemma \ref{regNL}). 

Now, we have that the odd reflection 
of the function $\overline \psi= \psi-\tilde \psi$ satisfies in  weak sense 
\begin{equation*}
\textrm{div\,} ( |y|^{-a} \nabla \overline \psi )=0\,\,\,\text{ in } B_{3R/2}\subset \R^{n+1}.
\end{equation*}
Hence by the results in \cite{FKS} (see Theorem \ref{HolderFKS} above), $\overline \psi$ is 
$C^{\beta}$ for some $\beta \in (0,1)$, depending only on $n$ and $s$
(perhaps different than the previous $\beta$). This and the above fact on $\tilde \psi$ give 
the desired result and estimates for $\psi=-y^au_y=\overline \psi + \tilde \psi$ 
\end{proof}

\subsection{Gradient estimates and integrability at infinity}
The following two results concern bounds for solutions
of problem \eqref{extAlpha}.  

\begin{propo}\label{integrability}
Let $f$ be a $C^{1,\gamma}(\RR)$ function with $\gamma >\max(0,1-2s)$
and $u \in L^{\infty}(\R_+^{n+1})$ a weak solution of problem
\eqref{extAlpha}. 

Then, $\nabla_x u$ and $y^a \partial_y u$ belong to $L^\infty(\RR^{n+1}_+)$. In addition, 
given $s_0>1/2$, there exists a constant $C_1$ depending only on $n$, $s_0$, $\|f\|_{C^{1,\gamma}}$ 
and $\|u\|_{L^\infty(\RR^{n+1}_+)}$, such that for every $s >s_0$, we have 

\begin{equation}\label{C1bound}
\|\nabla_x u\|_{L^\infty({\RR^{n+1}_+})} +
 (1+a)\|y^a \partial_y u \|_{L^\infty({\RR^{n+1}_+})}\leq C_1 .
 \end{equation}

Furthermore,  we have
\begin{equation}
\label{decay}
|\nabla u(x,y)|\leq \frac{C_2}{y}\qquad\text{for } y>0, 
\end{equation}
where the constant $C_2$ is uniformly bounded for $a \in (-1,1)$.

As a consequence of \eqref{C1bound}, we have 
$$y^a |\nabla u|^2 \in L_{{\rm loc}}^1(\overline{\RR^{n+1}_+}).$$
\end{propo}

\begin{proof}
The bound $\|\nabla_x u \|_{L^\infty({\RR^{n+1}_+})}$ in \eqref{C1bound} follows from 
Lemma \ref{regNL} and the fact that 
$$\frac{d_s}{2(1-s)}=\frac{d_s}{1+a}\to 1\,\,\,\,\,\mbox{as}\,\,\,s \to 1; $$
see Remark 3.11. 
The bound on $(1+a)\|y^a \partial_y u \|_{L^\infty({\RR^{n+1}_+})}$ in \eqref{C1bound} 
follows from duality (see Proposition 3.6) and the boundary condition
$$-(1+a)y^a \partial_y u |_{y=0}=f(u) \in L^\infty(\RR^n).$$

The last bound \eqref{decay} follows from rescaling the equation $L_au=0$ in $B_{y_0/2}(x_0,y_0)$ 
to the same equation for $\overline u (x',y')=u(x_0+y_0x',y_0y')$ in $B_{1/2}(0,1)$. 
Then we use that the operator $L_a$ is uniformly elliptic and has Lipschitz coeffcients 
$(y')^a$ with constants independent of $a \in (-1,1) $ ---since $1/2 < y'<3/2$ in this ball.   
\end{proof}

The following result is concerned with solutions of
\eqref{extAlpha} with limits in one Euclidean variable, or in
the radial variable, at infinity. 

\begin{rema}\label{monotlayer}
{\rm
If  $u$ is a layer solution of \eqref{extAlpha}, then not only $u_x(x,0)>0$ but $u_x(x,y) >0$ 
for every $y \geq 0$. Indeed, since $L_a u_x=0$ in $\RR^{n+1}_+$ and $u_x $ is bounded and 
continuous in $\overline{\RR^{n+1}_+}$ (by Lemma \ref{regNL}), Remark \ref{equipb} gives 
that $u_x(\cdot ,y)$ is the convolution of $P_s(\cdot ,y)$ with $u_x(\cdot,0) >0$. Hence the result follows.  
}
\end{rema}

\begin{lem}\label{infinit}
{\rm (i)} Let $u$ be a bounded solution of \eqref{extAlpha} such that 
\begin{equation}\label{asslim}
\displaystyle{\lim_{x_1  \rightarrow \pm \infty}} u(x,0)=L^{\pm}
\end{equation}
for every $(x_2,...,x_{n})\in \mathbb{R}^{n-1}$ and some constants $L^\pm$. 
Then,
\begin{equation}
\label{fpm}
f(L^+)=f(L^-)=0
\end{equation}  
and 
\begin{equation} \label{limuypos}
\displaystyle{\lim_{x_1  \rightarrow \pm \infty}} u(x,y)=L^{\pm}
\end{equation}
for every $(x_2,...,x_{n})\in \mathbb{R}^{n-1}$ and $y \geq 0$. 
Moreover, for every fixed $R>0$ and $(x_2,...,x_{n})\in \R^{n-1}$,
we have 
\begin{equation}
\label{inf1}
\|u-L^\pm\|_{L^\infty(B^+_R(x,0))} \rightarrow 0\,\,\mbox{as $x_1
  \rightarrow \pm \infty$},
\end{equation}
\begin{equation}
\label{inf2}
\|\nabla_x u\|_{L^\infty(B^+_R(x,0))} \rightarrow 0\,\,\mbox{as $x_1
  \rightarrow \pm \infty$},
\end{equation}
and
\begin{equation}
\label{inf3}
\|y^a u_y\|_{L^\infty(B^+_R(x,0))} \rightarrow 0\,\,\mbox{as $x_1
  \rightarrow \pm \infty$.}
\end{equation}

{\rm (ii)} Let $u$ be a radial solution of \eqref{extAlpha} such that 
\begin{equation}\label{asslimrad}
\displaystyle{\lim_{|x|  \rightarrow  \infty}} u(|x|,0)=0.
\end{equation}
Then,
\begin{equation}
\label{fpmrad}
f(0)=0. 
\end{equation}  
Moreover, for every fixed $R>0$, we have 
\begin{equation}
\label{inf2rad}
\|u\|_{L^\infty(B^+_R(x,0))}+\|\nabla_x u\|_{L^\infty(B^+_R(x,0))}+\|y^a u_y\|_{L^\infty(B^+_R(x,0))}  
\rightarrow 0\,\,\mbox{as $|x|
  \rightarrow \infty$}.
\end{equation}
\end{lem}

\begin{proof}
As in Lemma 2.4 in \cite{CSM} for the half-Laplacian, the lemma  
 follows easily by a compactness argument and the invariance
of the problem under translations in $x_1$. Indeed, in both cases (i)
and (ii) of the lemma, one considers the family
of translated (or slided) solutions in the $x_1$-variable.
In the radial case (ii) we proceed like this in each Euclidean variable, 
not only the $x_1$.

By the H\"older estimates of Lemma~\ref{regularity1}, the translated solutions
converge locally uniformly and up to subsequences, to a solution of the same problem
\eqref{extAlpha}. By assumption \eqref{asslim} or \eqref{asslimrad}, such limit is
identically constant. {From} this, \eqref{fpm}, \eqref{fpmrad}, \eqref{limuypos}, \eqref{inf1}
and \eqref{inf3} follow immediately. Finally, the $C^{\beta}$ estimate for $\nabla_x u$
of Lemma~\ref{regNL} leads to \eqref{inf2} or \eqref{inf2rad}.
\end{proof}

\subsection{A Harnack inequality}
The following Harnack inequality for linear Neumann problems
will be useful in the study of stable solutions
of~\eqref{problem}. 
\begin{lem}
\label{harnack}
Let $\varphi\in H^1(B_{4R}^+,y^a)$ be a nonnegative weak solution of
\begin{equation*} 
\label{linearized}
\begin{cases}
L_a \varphi = 0&\text{ in } B_{4R}^+\subset\R^{n+1}_+\\ 
\dfrac{\partial\varphi}{\partial\nu^a}+d(x)\varphi = 0&\text{ on }
\Gamma^0_{4R} ,
\end{cases}
\end{equation*}
where $d$ is a bounded function in $\Gamma_{4R}^0$. Then,
\begin{equation}\label{harnackineq}
\sup_{B_R^+} \varphi \le C_R\, \inf_{B_R^+} \varphi  ,
\end{equation}
for some constant $C_R$ depending only on $n$, $a$, and 
$R^{1-a}  \|d\|_{L^\infty(\Gamma_{4R}^0)}$.
\end{lem}

\begin{proof}
By scaling, one can assume $R=1$. We introduce the new function
\begin{equation*}
\Psi^A(x,y)=e^{Ay^{1-a}}\varphi(x,y). 
\end{equation*}
The function $\Psi^A$ satisfies 
\begin{equation*}
\begin{cases}
\textrm{div\,}  (y^a \nabla (e^{-Ay^{1-a}} \Psi^A))   = 0&\text{ in } B_{4}^+\\
\dfrac{\partial\Psi^A}{\partial\nu^a}= -(A(1-a)+d(x))\Psi^A&\text{ on }
\Gamma^0_{4} .
\end{cases}
\end{equation*}
Therefore, choosing 
$$
A = \frac{\|d\|_{L^\infty(\Gamma_{4}^0)}}{1-a},
$$
we have $\partial_{\nu^a}  \Psi^A \leq 0$ on $\Gamma^0_{4}$. 
We consider the even extension of $\Psi^A$ across $\Gamma^0_{4}$,
defined by
$$
\tilde\Psi^A(x,y)= \Psi^A(x,-y)\quad \mbox{ for } (x,y)\in B_4, y\le 0.
$$ 
Since $\partial_{\nu^a}  \Psi^A \leq 0$ on $\Gamma^0_{4}$, $\tilde\Psi^A$ satisfies
\begin{equation*}
-\textrm{div\,} (|y|^a \nabla (e^{-A|y|^{1-a}} \tilde \Psi^A)) \le 0  
\quad \text{in } B_4 
\end{equation*}
in the weak sense. 

Next, taking $-A$ we obtain $\partial_{\nu^a}  \Psi^{-A} \geq 0$ on $\Gamma^0_{4}$,
and arguing as before we deduce that $\tilde\Psi^{-A}$ satisfies
\begin{equation*}
-\textrm{div\,} (|y|^a  \nabla (e^{A|y|^{1-a}} \tilde \Psi^{-A}))\ge 0  
\quad \text{in } B_4
\end{equation*}
in the weak sense. 

Denote by $\mathcal L_{A}$ the operator
$$\mathcal L_{A} w:= \textrm{div\,} (|y|^a  \nabla (e^{-A|y|^{1-a}} w)).$$
We introduce now the solutions $h^{\pm A}$ of 
$$
\begin{cases}  
\mathcal L_{\pm A} h^{\pm A}=0&\quad\text{in $B_4$}\\
h^{\pm A}=\tilde\Psi^{\pm A}&\quad\text{on $\partial B_4$.}
\end{cases}
$$
These solutions are obtained from the solutions of the Dirichlet problem for $L_a$ 
given by Theorem 3.2, after multiplying $h^{\pm A}$ by $e^{\pm A|y|^{1-a}}$.
By the weak maximum principle and the previous considerations,
we have that
\begin{equation}
\label{weak
comp1}
\tilde\Psi^{A}\leq h^{A} \quad\text{and}\quad h^{-A}\leq\tilde\Psi^{-A}
\quad\text{ in } B_4.
\end{equation}

On the other hand, since $\Psi^{A}/\Psi^{-A}=e^{2Ay^{1-a}}\leq
e^{2A4^{1-a}}\leq e^{32A}$ in $B_4^+$, we have that  
$\tilde\Psi^{A}\leq e^{32A}\tilde\Psi^{-A}$ 
on $\partial B_4$. Next, since $\mathcal L_{A} h^{A}=0=\mathcal L_{-A} h^{-A}
=\mathcal L_{A}(e^{2A|y|^{1-a}}h^{-A})$ and on the boundary $\partial B_4$,
$h^{A}=\tilde\Psi^{A} \leq  e^{32A}\tilde\Psi^{-A}
\leq  e^{32A}e^{2A|y|^{1-a}}\tilde\Psi^{-A}= e^{32A}\left \{e^{2A|y|^{1-a}}h^{-A}\right \}$,
the weak maximum principle for the operator $\mathcal L_{A}$ leads to
\begin{equation}
\label{weakcomp2}
h^{A}\leq e^{32A} \left \{e^{2A|y|^{1-a}}h^{-A} \right \}\leq e^{64A} h^{-A}\quad\text{in } B_4.
\end{equation}

Next, note that $L_a(e^{-A|y|^{1-a}}h^A)=0$ in $B_4$. According to the Harnack inequality 
of Fabes-Kenig-Serapioni, Lemma~2.3.5 of \cite{FKS} (Theorem 3.4 above), we deduce that 
$$
\sup_{B_1}(e^{-A|y|^{1-a}}h^A)\leq C \inf_{B_1}(e^{-A|y|^{1-a}}h^A)
$$
for some constant $C$ depending only on $n$ and $a$. Thus,
\begin{equation}
\label{weakcomp3}
\sup_{B_1^+}h^A \leq C e^{A} \inf_{B_1^+}h^A.
\end{equation}

Using \eqref{weakcomp1} and  \eqref{weakcomp2}, we deduce
\begin{equation}
\label{weakcomp4}
\varphi \leq \tilde\Psi^A \leq h^A \leq e^{64A} 
h^{-A} \leq e^{64A} \tilde\Psi^{-A} \leq e^{64A} \varphi
\quad\text{ in } B_1^+.
\end{equation}
Finally, \eqref{weakcomp4} and \eqref{weakcomp3}
lead immediately to the desired result. 
\end{proof}

\subsection{A Liouville theorem}
We prove the following theorem, which will be useful in \cite{CS2} to
prove a symmetry result for stable solutions of \eqref{extAlpha} in $\mathbb{R}^2_+$ . 
This is a generalization to degenerate elliptic equations of the Liouville
theorem given in \cite{CSM}. This type of result had been already used
in \cite{AC} to prove the De Giorgi conjecture for reactions in the
interior in three dimensional spaces. 
 
\begin{theorem}\label{liouville}
Let $\varphi \in L^{\infty}_{\rm loc}(\overline{\mathbb{R}_+^{n+1}})$ be a positive function. 
Suppose that $\sigma \in
H^1_{\rm loc}(\overline{\mathbb{R}_+^{n+1}}, y^a)$
is such that
\begin{equation}\label{liouveq}
\begin{cases}
-\sigma {\rm div\,}(y^a \varphi^2 \nabla\sigma) \leq 0
&\quad \hbox{in } \R^{n+1}_+\\
-\sigma y^a \partial_y \sigma  \leq 0
&\quad \hbox{on } \partial\R^{n+1}_+
\end{cases}
\end{equation}
in the weak sense. Assume that for every $R >1$, 
\begin{equation}\label{intR2}
\int_{B^+_R} y^{a} (\sigma \varphi)^2 \;dxdy\leq C R^2
\end{equation}
for some constant $C$ independent of $R$. 

Then $\sigma$ is constant. 
\end{theorem}
\begin{proof}
We adapt the proof given in \cite{CSM}. Let $\zeta$ be a $C^\infty$ function on $[0,+\infty)$ 
such that
$0\le\zeta\le 1$ and  
$$
\zeta = \begin{cases}  1&\quad\text{for $0\le t\le 1$,}\\
                0&\quad\text{for $t\ge 2$.}
\end{cases}
$$
For $R>1$ and $(x,y)\in\R^{n+1}_{+}$, let
$\zeta_R(x,y)=\zeta \left( r / R\right)$, where $r=\vert(x,y)\vert$.

Multiplying \eqref{liouveq} by $\zeta_R^2$ and integrating by parts
in $\R^{n+1}_+$, we obtain
\begin{eqnarray*}
& & \hspace{-2cm}
\int_{\R^{n+1}_+} y^a \zeta_R^2\, \varphi^2 \vert\nabla\sigma\vert^2\;dxdy 
\leq 
-2 \int_{\R^{n+1}_+}y^a \zeta_R\, \varphi^2 \sigma\, \nabla\zeta_R\, \nabla\sigma \;dxdy\\
\ \hspace{-1cm}
& \leq  & 2 \left [ \int_{\R^{n+1}_+\cap\{R<r<2R \}} 
 y^a \zeta_R^2\, \varphi^2  \vert\nabla\sigma\vert^2\;dxdy \right ]^{1/2} \cdot
\\
\hspace{1cm}
& & \qquad \cdot \left [ \int_{\R^{n+1}_+} y^a\varphi^2  \sigma^2 
\vert\nabla\zeta_R\vert^2 \;dxdy
\right ]^{1/2} \\
\hspace{-1cm}
& \leq  & C \left [ \int_{\R^{n+1}_+\cap\{R<r<2R \}} 
 y^a \zeta_R^2\,\varphi^2 \vert\nabla\sigma\vert^2 \;dxdy\right ]^{1/2}
\cdot
\\
\hspace{1cm}
& & \qquad \cdot
\left [\frac{1}{R^2} \int_{B_{2R}^+} y^{a} (\varphi\sigma)^2 \;dxdy
\right ]^{1/2},
\end{eqnarray*}
for some constant $C$ independent of $R$. Using hypothesis
\eqref{intR2}, we infer that 
\begin{multline}\label{boundann}
\int_{\R^{n+1}_+}y^a \zeta_R^2\, \varphi^2 \vert\nabla\sigma\vert^2\;dxdy \le \\
C \left [ \int_{\R^{n+1}_+\cap\{R<r<2R \}} 
y^a \zeta_R^2\,  \varphi^2 \vert\nabla\sigma\vert^2\;dxdy \right ]^{1/2} ,
\end{multline}
again with $C$ independent of $R$. Hence,
$\int_{\R^{n+1}_+}  y^a \zeta_R^2\, \varphi^2  \vert\nabla\sigma\vert^2\;dxdy \le C$ and,
letting $R\rightarrow\infty$, we deduce
$\int_{\R^{n+1}_+}y^a \varphi^2  \vert\nabla\sigma\vert^2\;dxdy \le C$.
It follows that the right hand side of \eqref{boundann} tends to
zero as $R\rightarrow\infty$, and therefore
$\int_{\R^{n+1}_+} y^a \varphi^2 \vert\nabla\sigma\vert^2\;dxdy =0$.
We conclude that $\sigma$ is constant.
\end{proof}

\subsection{A Hopf principle}
The following proposition provides a Hopf boundary lemma in our context. 
\begin{propo}\label{hopf}
Let $a \in (-1,1)$ and consider the cylinder 
$\mathcal C_{R,1}= \Gamma_R^0 \times (0,1) \subset \RR^{n+1}_+$ where $\Gamma_R^0$ 
is the ball of center $0$ and radius $R$ in $\RR^n$. Let $u \in C(\overline{\mathcal C_{R,1}}) 
\cap H^1(\mathcal C_{R,1},y^a)$ satisfy  
\begin{equation*}
\begin{cases}
L_a u \leq 0&\text{ in } \mathcal C_{R,1} \\
u> 0&\text{ in } \mathcal C_{R,1} \\ 
u(0,0)=0.&
\end{cases}
\end{equation*}

Then, 
$$\limsup_{y \rightarrow 0^+ }-y^a \frac{u(0,y)}{y}<0.$$
In addition, if $y^a u_y \in C(\overline{\mathcal C_{R,1}})$, then 
$$\partial_{\nu^a} u (0,0) <0. $$ 
\end{propo}
\begin{proof}

Consider the function on $\mathcal C_{R,1}$ defined by 
\begin{equation*}
w_A(x,y)=y^{-a}(y+Ay^2)\varphi(x),
\end{equation*}
where $A$ is a constant to be chosen later and $\varphi=\varphi(x)$ is the first 
eigenfunction of $-\Delta_x$ in $\Gamma^0_{R/2}$ with Dirichlet boundary conditions, i.e., 
\begin{equation*}
\begin{cases}
-\Delta_x \varphi=\lambda_1 \varphi & \text{in } \Gamma_{R/2}^0\\
\varphi=0 &\text{on } \partial \Gamma_{R/2}^0.\\
\end{cases}
\end{equation*}
Notice that $\lambda_1>0$ and that we can choose $\varphi >0$ in $\Gamma_{R/2}^0$
with $\|\varphi\|_{L^\infty}=1$. The function $w_A$ satisfies
\begin{equation*}
\begin{cases}
L_a w_A=\varphi(x) \Big \{A(2-a) -\lambda_1 (y+Ay^2) \Big \} & \mbox{in } 
\mathcal C_{R/2,1}\\
w_A\geq 0& \mbox{in } \overline{\mathcal C_{R/2,1}}\\
w_A=0& \mbox{on } \partial \Gamma_{R/2}^0 \times [0,1). 
\end{cases}
\end{equation*}  
Therefore, choosing $A$ large enough, we have in $\mathcal C_{R/2,1}$
$$L_a w_A  \geq 0.$$

Hence, for $\varepsilon >0$,  
$$L_a (u-\varepsilon w_A)\leq 0\,\,\,\mbox{in} \,\,\mathcal C_{R/2,1}$$
and $u-\varepsilon w_A= u \geq 0$ on $\partial 
\Gamma_{R/2}^0 \times [0,1)$. Moreover, taking $\varepsilon >0$ small enough, we have on 
$\Gamma_{R/2}^0 \times \left \{y=1/2 \right \}$
$$u \geq \varepsilon  w_A,$$
since $u$ is continuous and positive on the closure of this set. Notice furthermore that 
$w_A=0$ on $\Gamma_R^0 \times \left \{y=0 \right \}$. 
Thus, we have 
\begin{equation*}
\begin{cases}
L_a(u-\varepsilon w_A) \leq 0&\mbox{in } \mathcal C_{R/2,1/2}\\
u-\varepsilon w_A \geq 0&\mbox{on } \partial \mathcal C_{R/2,1/2}.
\end{cases}
\end{equation*}

The weak maximum principle then implies that in $\overline {\mathcal C_{R/2,1/2}}$
$$u-\varepsilon w_A\geq 0 .$$
  Consequently, this leads to 
  $$\limsup_{y \rightarrow 0^+ }-y^a \frac{u(0,y)}{y}\leq \varepsilon 
\limsup_{y \rightarrow 0^+ }-y^a\frac{w_A(0,y)}{y}=-\varepsilon \varphi(0)<0,$$
as claimed in the proposition. 

Assume, in addition, $y^au_y \in C(\overline{\mathcal C}_{R,1})$. Let $y_0 \leq 1/2$. Since 
$(u-\varepsilon w_A)(0,\cdot) \geq 0 $ in $[0,y_0]$ and $(u-\varepsilon w_A)(0,0) = 0$, 
we have $(u_y-\varepsilon(w_A)_y)(0,y_1) \geq 0$ for some $y_1 \in (0,y_0)$. Repeating 
this argument for a sequence of $y_0's$ tending to $0$, we conclude that 
$-y^au_y \leq -\varepsilon y^a(w_A)_y$ at a sequence of points $(0,y_j)$ with $y_j \downarrow 0$. 
Since we assume $y^au_y$ continuous up to $\left \{y=0 \right \}$ and $-\varepsilon (y^a(w_A)_y)(0,y_j) 
\to -\varepsilon \varphi(0) $, we conclude that $\partial_{\nu^a}u(0,0)<0.$
\end{proof}

\begin{coro}
Let $a \in (-1,1)$ and $\varepsilon >0$. Let $d$ be a H\"older continuous function in 
$\Gamma^0_\varepsilon$ and $u \in L^\infty(B^+_\varepsilon) \cap H^1(B^+_\varepsilon,y^a)$ 
be a weak solution of 
\begin{equation*}
\begin{cases}
L_a u =0&\text{ in } B^+_\varepsilon \\
u \geq  0&\text{ in }  B^+_\varepsilon\\ 
\partial_{\nu^a} u +d(x) u=0&\text{ on } \Gamma^0_\varepsilon .  
\end{cases}
\end{equation*}
Then, $u >0$ in $B^+_\varepsilon \cup \Gamma^0_\varepsilon$ unless $u \equiv 0$ in $B^+_\varepsilon. $
\end{coro}

\begin{proof}
We apply Lemma \ref{regularity1} to obtain that $u$ (for this, see the proof of
the lemma) and $y^au_y$ are $C^\alpha$ up to 
the boundary. Hence the equation 
\begin{equation}\label{tempHpf}
\partial_{\nu^a} u +d(x) u=0
\end{equation}
is satisfied pointwise on $\Gamma^0_\varepsilon$. If $u$ is not identically $0$ in 
$B^+_\varepsilon$ then $u>0$ in $B^+_\varepsilon$ by the strong maximum principle 
for the operator $L_a$. Now, if $u(x_0,0)=0$ at some point $(x_0,0)\in \Gamma^0_\varepsilon$, 
then a rescaled version of Proposition \ref{hopf} gives $\partial_{\nu^a} u (x_0,0) <0$. 
This contradicts \eqref{tempHpf}. 
\end{proof}

\subsection{A maximum principle}
Here we present a maximum principle
related to the operator $L_a$ and to the fractional Laplacian.
We will need it in our subsequent article to prove monotonicity properties
for solutions in $\RR $ with limits, as well as the uniqueness (up to translations) 
of layer solutions in $\RR$.
Recall that section~3 already contained some Liouville and maximum
principles for these operators.

\begin{lem}
\label{l3} 
Let $u\in (C\cap L^\infty)(\overline{\R^{n+1}_+})$ with $y^a u_y \in C(\overline{\RR^{n+1}_+})$
satisfy 
\begin{equation} \label{maxpr}
\begin{cases}
L_a u= 0&\text{ in } \R^{n+1}_+\\ 
\dfrac{\partial u}{\partial\nu^a}+d(x)u\ge 0&\text{ on }\partial\R^{n+1}_+,
\end{cases}
\end{equation}
where $d$ is a bounded function,
and also
\begin{equation}
\label{lim0}
u(x,0)\rightarrow 0 \qquad \mbox{as } |x|\rightarrow\infty .
\end{equation}
Assume that there exists a nonempty set $H\subset\R^n$ such that
$u(x,0)> 0$ for $x\in H$, and $d(x)\ge 0$ for $x\not\in H$.

Then, $u>0$ in $\overline{\R^{n+1}_+}$.
\end{lem}

\begin{proof}
By Remark~\ref{equipb} applied to $v(x)=u(x,0)-\inf_{\R^n} u(\cdot,0)\geq 0$,
we see that $u-\inf_{\R^n} u(\cdot,0)\geq 0$ in $\R^{n+1}_+$.
Thus, $\inf_{\R^{n+1}_+} u =\inf_{\R^n} u(\cdot,0)$.

Arguing by contradiction, assume that there exists a point $(x_0,y_0)$ in 
$\overline{\RR^{n+1}_+}$ such that $u(x_0,y_0) \leq 0$. 
Then, in case $\inf_{\RR^{n+1}_+}u=0$,
the minimum of $u$ is a achieved at $(x_0,y_0)$. In case $\inf_{\RR^{n+1}_+}u=
\inf_{\R^n} u(\cdot,0)<0$,
using that $u(x,0) \rightarrow 0$ as $|x| \rightarrow +\infty$, there exists a point 
$(x_1,0)$ at which the minimum of $u$ is achieved. 
In both cases we conclude that the nonpositive minimum of $u$ is achieved at a point
$(x_2,y_2)$.

By the strong maximum principle, we cannot have $y_2>0$, since $u$ is not
identically constant (recall $u(\cdot,0)> 0$ in $H\not = \emptyset$).
Thus $y_2=0$. According to the Hopf lemma \ref{hopf} and since $y^a u_y \in 
C(\overline{\RR^{n+1}_+})$, we have 
\begin{equation*}
\frac{\partial u}{\partial \nu^a}(0,y_2) <0. 
\end{equation*}   
Since $u(x_2,0)\leq 0$ then $x_2 \notin H$, and thus we have $d(x_2)\geq 0$. 
Now, using the boundary condition in \eqref{maxpr} at $x=x_2$, 
we reach a contradiction.
\end{proof}

\begin{rema}
{\rm
Lemma \ref{l3} can be stated in an equivalent way using the equation
$$
(-\Delta)^s v +d(x) v \geq 0 \qquad \text{in }\R^n
$$
and assuming the same conditions on $v$ as those for $u(\cdot, 0)$ in the previous lemma.
In addition, an alternative proof of the lemma can be given using
the integral expression \eqref{fracpv} for $(-\Delta)^s v(x_2)$,
that will be negative at a point of minimum (since $v$ is not
identically constant in the proof).
}
\end{rema}

\section{Hamiltonian estimates}

This section is devoted to establish the main facts needed to prove 
Theorems~\ref{modthm} and \ref{modthmGS}. We start with an easy
lemma that will be needed later in several occasions.

\begin{lem}\label{haminf}
Let $u\in L^\infty(\R^{n+1}_+)$ be a bounded solution of \eqref{extAlpha}. Then, 
for all $x\in\R^n$, we have
$\int_0^{+\infty}t^{a}|\nabla u(x,t)|^2 dt<\infty.$
In addition,  the integral can be differentiated with respect to $x\in\R^n$ under
the integral sign. 
Furthermore,
\begin{equation}\label{lim0ymod}
\lim_{M\to +\infty} \int_M^{+\infty}t^{a}|\nabla u(x,t)|^2 dt = 0
\end{equation}
uniformly in $x\in \R^n$.

If in addition, $u$ is either a layer solution in $\RR$ (here $n=1$) or
$u$ is a radial solution in $\RR^n$ for which $\lim_{|x|\to\infty} u(|x|,0)$
exists, then
\begin{equation}\label{lim0xmod}
\lim_{|x|\to\infty} \int_0^{+\infty}t^{a}|\nabla u(x,t)|^2 dt = 0.
\end{equation}
\end{lem}

\begin{proof}
The first two statements and \eqref{lim0ymod} follow directly from the gradient 
bounds in Proposition \ref{integrability}. The statement \eqref{lim0xmod} is a consequence 
of \eqref{lim0ymod} and of  Lemma \ref{infinit}.
\end{proof}

\subsection{Hamiltonian equality and estimate for layer solutions}

This subsection contains two lemmas. The first one establishes that 
the Hamiltonian is conserved for layer solutions in dimension one.

\begin{lem}\label{hamilton}
Let $n=1$ and assume that $u$ is a layer solution of \eqref{extAlpha}. Then, 
for all $x\in\R$ we have
$\int_0^{+\infty}t^{a}|\nabla u(x,t)|^2 dt<\infty$ and the Hamiltonian identity
\begin{equation}\label{mmequaltemp}
(1+a)\int_0^{+\infty} \frac{t^a}{2}  \left\{u_x^2(x,t)-u_y^2(x,t)\right\} dt=G(u(x,0))-G(1).
\end{equation}
As a consequence,
\begin{equation}\label{Gequal}
G(1)=G(-1).
\end{equation}
\end{lem}

\begin{proof}
The  integrability of $t^a |\nabla u(x,t)|^2$ follows from Lemma 5.1.  
We now establish equality \eqref{mmequaltemp}. It will be crucial  that the weight in $L_a$ does
not depend on the tangential variable $x$.

Following \cite{CSM}, we consider the function
\begin{equation}
\label{definham}
v(x)=\int_0^{+\infty}\dfrac{t^a }{2}\left\{
u_x^2(x,t)-u_y^2(x,t)\right\} dt .  
\end{equation}
Lemma \ref{haminf} allows us to differentiate under the integral in
\eqref{definham} to get 
\begin{equation*}
\dfrac{d}{dx}v(x) =
\int_0^{+\infty}t^a ( u_{xx}u_x-u_{xy}u_y) (x,t)dt. 
\end{equation*}
Noticing that 
\begin{equation*}
L_a u= \partial_y (y^a u_y)+y^a u_{xx} =0
\end{equation*}
and after an integration by parts (which is justified by Lemma \ref{haminf}) we have
\begin{equation*}
\dfrac{d}{dx}v(x) =
\lim_{y \rightarrow  0^+} y^a u_y(x,y) u_x(x,y)=
\frac{1}{1+a}\frac{d}{dx}G(u(x,0)). 
\end{equation*}

The function $(1+a)v(x)-\{G(u(x,0))-G(1)\}$ is then constant in $x$.
Letting $x\to +\infty$ and using Lemma \ref{haminf}, we
have that this constant is actually zero. Letting now
$x\to -\infty$ and using Lemma \ref{haminf}, we deduce 
$G(1)=G(-1)$.
\end{proof}

We have obtained that a necessary condition for the 
existence of a layer solution in $\RR$ is that 
$G(1)=G(-1)$. 
The other necessary condition will follow from the following result
---our Modica-type estimate for layer solutions in dimension $1$ (Theorem \ref{modthm}). 

\begin{lem}\label{modica1}
Let $n=1$ and assume that $u$ is a layer solution of \eqref{extAlpha}. 
Then, for every $y \geq 0$ and all $x \in \RR$, we have
\begin{equation}\label{modicatype}
(1+a)\int_0^y \frac{t^a}{2} \left\{u_x^2(x,t)-u_y^2(x,t)\right\} \,dt < G(u(x,0))-G(1). 
\end{equation}
\end{lem}

\begin{proof}
We introduce the function
\begin{equation*}\label{auxiliaryw}
w(x,y)=\int_0^y\dfrac{t^a}{2}\left\{
u_x^2(x,t)-u_y^2(x,t)\right\} dt ,
\end{equation*}
which is bounded in all $\R^2_+$ by Lemma
\ref{haminf}. We introduce the function 
\begin{equation*}
\overline{w}(x,y)=\frac{1}{1+a}\Big \{G(u(x,0))-G(1) \Big \}-w(x,y). 
\end{equation*}
The function $\overline{w}$ is bounded in $\R^2_+$ and we need to show that
$\overline{w}>0$ in $\overline{\R^2_+}$. 

We first derive some equations for $\overline{w}$ 
which will be useful in the sequel. We have, for all $y>0$,
\begin{equation}\label{eq}
\overline{w}_y(x,y)=-\frac{y^a}{2}(u^2_x(x,y)-u^2_y(x,y)). 
\end{equation}
Furthermore, using $L_a u=0$ and integrating by parts as in the previous proof, one gets for all $y >0$
\begin{equation}\label{tempM}
\overline{w}_x(x,y)=y^a u_x(x,y)u_y(x,y). 
\end{equation}
Using the two previous equalities and  the equation $L_a u=0$, we have for all $y>0$
\begin{equation}\label{tempM1}
L_a \overline{w}=-a \, y^{2a-1}u^2_x
\end{equation}
and 
\begin{equation}\label{tempM2}
L_{-a} \overline{w}=-a \, y^{-1}u^2_y. 
\end{equation}

We claim that $\overline w$ does not achieve its infimum at a point in $\overline {\RR^2_+}$. 
We assume the contrary and  reach a contradiction. Let $(x_0,y_0)$ be a point where the infimum 
is achieved. There are now two cases depending if $y_0$ is on the boundary or not. We will also 
use that $\overline w$ is not identically constant. Indeed, if it were, since $w(\cdot,0)\equiv 0$  
then 
$$\mbox{constant}=\,\,\overline w(\cdot,0)=\frac{1}{1+a} \Big \{ G(u(\cdot,0))-G(1) \Big \}.$$
Thus $G$ is constant in $(-1,1)$, $f \equiv 0$ in $(-1,1)$ and $u$ is a bounded function satisfying 
\eqref{extAlpha} with $f \equiv 0$. Hence, after an even reflection across $\left \{ y=0 \right \}$, 
Theorem 3.4 ensures that  $u$ is a constant, a contradition with $u_x >0$. 

{\it Case 1: $y_0=0$}.  After a translation in $x$, we may assume $x_0=0$. Since $x_0=0$ is a global 
minimum of $(1+a)\overline w (\cdot,0)=G(u(\cdot,0))-G(1)$, we have 
$$
0=(d/dx)G(u(x,0))|_{x=0}=
-f(u(0,0))u_x(0,0),
$$
and therefore
\begin{equation}
\label{ux0}
0=-f(u(0,0))=(1+a)\lim_{ y \rightarrow 0^+} y^{a} u_y(0,y),
\end{equation}
since $u$ is a layer solution (i.e., $u_x(x,0)>0$). 
For every $(x,y) \in \overline{\RR^2_+}$, we have 
 by Remark \ref{monotlayer},
\begin{equation}\label{5**}
u_x(x,y) >0\,\,\,\mbox{in}\,\,\overline{\RR^2_+}.
\end{equation}

We now divide the conclusion into two subcases. Let consider first the case $a \geq 0$. By \eqref{eq}, 
we see that $y^a\overline w_y$ is H\"older continuous up to $y=0$. Since $L_a \overline w \leq 0$ 
by \eqref{tempM1} and $\overline w$ is not identically a constant, we have 
$\overline w > \overline w(0,0)$ in $\RR^2_+$. Thus  the Hopf principle (see Proposition \ref{hopf}) 
gives that 
$$0 > -\lim_{y \to 0^+} y^a \overline w_y (0,y).$$
Now, using \eqref{eq} and \eqref{ux0}, we have 
\begin{eqnarray*} 
0 &  > & -\lim_{y \rightarrow 0^+} y^{a} \overline w_y(0,y) \\ 
& = & \lim_{y \rightarrow 0^+}  \dfrac{y^{2a}}{2}\{u_x^2(0,y)-
u_y^2(0,y)\}\\
& = & \lim_{y \rightarrow 0^+} \dfrac{y^{2a}}{2}u^2_x(0,y)\geq 0,
\end{eqnarray*} 
a contradiction. 

We turn now to the case $a <0$. Since  $(0,0)$ is a global 
minimum for $\overline w (x,y)$, one gets
\begin{eqnarray*} 
0 &\geq & \liminf_{y \rightarrow 0^+}- y ^{-a} \overline w_y (0,y) \\
& = & \liminf_{y \rightarrow 0^+}\frac{1}{2}\Big ( u_x^2(0,y)-
u_y^2(0,y) \Big )=\frac12 u_x^2(0,0)>0,
\end{eqnarray*} 
a contradiction. We have used that, by Lemma \ref{regularity1} 
$|u_y (0,y)| \leq C y^{-a} \to 0$ as $y \to 0^+$.  

{\it Case 2: $y_0>0$}. By \eqref{5**}, we have $u_x>0$ in $\overline{\RR^2_+}$. Using 
\eqref{tempM} and \eqref{tempM2}, we obtain 
$$0=L_{-a} \overline w +a y^{-1} u_y^2=L_{-a}\overline w +(ay^{-1-a} \frac{u_y}{u_x}) \overline w_x $$ 
which is the same as  
$$0= \nabla \cdot (y^{-a} \nabla \overline w)+b(x,y) \overline w_x \,\,\,\,\mbox{in} \,\RR^{2}_+, $$
with $b(x,y):=ay^{-1-a} u_y u_x^{-1}.$
But this last operator is uniformly elliptic with continuous coefficients in compact sets of 
$\left \{ y>0 \right \}$. Thus it cannot achieve its minimum at $(x_0,y_0), $ since $y_0>0$ 
and we have proved that $\overline w$ is not identically constant. 

Therefore, we now know that $\overline w$ cannot achieve its infimum at a point in $\overline {\RR^2_+}.$ 
To finish the proof, assume first 
$$\inf_{ \overline {\RR^{2}_+}} \overline w <0.$$
By Lemma \ref{hamilton}, $\overline w (x,y) \to 0$ as $y \to +\infty$ locally uniformly in $x$. 
By Lemma \ref{haminf}, we have $\overline{w}(x,y) \rightarrow 0$ as $|x| \rightarrow
+\infty$ uniformly in $y$. Therefore,  the infimum of $\overline w$ being negative, it should be 
achieved at a point in $\overline{\RR^2_+}$, a contradiction with what we have proven. Therefore,
$$\inf_{\overline{\RR^2_+}} \overline w \geq 0,$$
i.e. $\overline w \geq 0. $

Thus, if $\overline w $ vanished at some point in $\overline{\RR^2_+}$, this point would achieve 
the infimum of $\overline w$, a contradiction. Hence $w >0$ in $\overline{\RR^2_+}$ as stated 
in the lemma.    
\end{proof}

\subsection{The Hamiltonian for radial solutions}

The next lemma deals with bounded radial solutions $u$ of \eqref{extAlpha}.
Here we do not assume  $u$ to have a limit  at infinity.

\begin{lem}\label{hamrad}
Let $u$ be a bounded solution of \eqref{extAlpha}.
Assume that $u=u(|x|,y)$ is radially symmetric in $x$. Then, 
\begin{equation}\label{hamradially}
(1+a)\int_0^{+\infty} \frac{t^a}{2}\Big \{ {u}_r^2(r,t)- {u}_y^2(r,t) 
\Big\}\, dt-G({u}(r,0))
\end{equation}
is a nonincreasing function of $r$.
\end{lem}

\begin{proof}
The function ${u}$ solves the 
 
\begin{equation*}
\begin{cases}
 {u}_{rr}+\frac{n-1}{r} {u}_{r}+{u}_{yy}+\frac{a}{y} {u}_{y}=0&\,\,\,\,
\mbox{in $(0,+\infty) \times (0,+\infty)$}\\
-(1+a)y^a {u}_y=f({u})&\,\,\,\,\mbox{on $(0,+\infty)\times\{y=0\}$}.
\end{cases}
\end{equation*}  

Let 
$$w(r):=\int_0^{+\infty} \frac{t^a}{2}\Big \{ {u}_r^2(r,t)- {u}_y^2(r,t) 
\Big\}\, dt.$$
By Lemma~\ref{haminf}, 
we can differentiate under the integral with respect to 
$r$ and obtain
$$\frac{d w(r)}{dr}=\int_0^{+\infty}  t^a \left\{ {u}_r {u}_{rr}-{u}_y{u}_{ry} 
\right\}(r,t)\, dt.$$
Performing one integration by parts and using the equation and Lemma \ref{haminf}, we end up with
$$\frac{d w(r)}{dr}=-\frac{n-1}{r} \int_0^{+\infty} t^a {u}^2_r(r,t)\, dt
+\frac{1}{1+a}G'({u}(r,0)) {u}_r(r,0).$$
As a consequence, the function 
$$(1+a)w(r)-G({u}(r,0))$$
is nonincreasing in $r$, as claimed. Furthermore, 
\begin{equation}\label{derivradial}
\frac{d}{dr} \left \{ (1+a)w(r)-G(u(r,0)) \right \}=-(1+a)\frac{n-1}{r}\int_0^\infty t^a u^2_r(r,t)\,dt. 
\end{equation}
\end{proof}

\section{The limit $s\to 1$ and the classical Laplacian}

In the following, we investigate the asymptotic $s \rightarrow 1$. 
For this, we will use crucially the previous Hamiltonian estimates.
We prove the following theorem.

\begin{theorem}\label{asymps}
Assume that $f\in C^{1,\gamma}(\RR)$ for some $\gamma \in (0,1)$
and that $\left \{ v_s \right \}$, with $s=s_k \in(0,1)$ and $s_k \uparrow 1$, 
is a sequence of layer solutions of 
$$(-\partial_{xx})^s v_s =f(v_s)\,\,\,\mbox{in } \RR,$$
such that $v_s(0)=0.$ Then, there exits a function $\overline v$ such that 
$$\lim_{s \uparrow 1} v_s=\overline v$$
in the uniform $C^2$ convergence on every compact set of $\RR$. 

Furthermore, the function $\overline v$ is the layer solution of
$$ -\overline v ''=f(\overline v)\,\,\,\,\,\mbox{in } \RR$$
with $\overline v(0)=0$, and satisfies the Hamiltonian equality 
\begin{equation}\label{hamilteq}
\frac12 (  \overline v')^2=G(\overline v)-G(1)\,\,\,\,\,\mbox{in } \RR. 
\end{equation}
\end{theorem}

The previous theorem is stronger than just saying that the limit when $s$ goes to $1$
 is a solution of an ODE, since it states that the limit is actually a layer itself. 
We can see Theorem \ref{asymps} as a stability result in the class of layer solutions of 
nonlocal (and local) equations. 

\begin{proof}[Proof of Theorem \ref{asymps}]
Let $v_s$ be a layer solution of 
$$(-\partial_{xx})^s v_s =f(v_s)\,\,\,\mbox{in } \RR $$
with $v_s(0)=0. $
Then, the extension $u_a=P_s \,* v_s$ of $v_s$ satisfies 
\begin{equation*}
\begin{cases}
\textrm{div\,}  (y^a \nabla u_a) =0&\,\,\,\,\mbox{in $\RR^2_+$}\\
(1+a) \frac{\partial u_a}{\partial \nu^a}=c_af(u_a)&\,\,\,\,\mbox{on $\partial \RR^2_+$}\\
u_a=v_s&\,\,\,\,\mbox{on $\partial \RR^2_+$},
\end{cases}
\end{equation*}  
where $c_a= d_s^{-1}(1+a)=d_s^{-1}2(1-s)$ and $d_s $ is the constant in Theorem  \ref{realization} 
and Remark  \ref{constant}. By \eqref{limctts}, we know that $c_a $ tends to $1$ as $a$ goes to $-1$. 
The weak formulation of this problem is 

\begin{equation}
(1+a)\int_{\RR^{2}_+} y^{a} \nabla u_a \cdot  \nabla \xi -\int_{\RR} c_a f(u_a)
\xi =0
\end{equation}
for all $\xi \in C^1(\overline{\RR^{2}_+})$ compactly supported. 

First notice that, 
by the regularity result in Lemma \ref{regNL}, which is uniform as $s \uparrow 1$, the functions 
$u_a$ and $\partial_x u_a$ converge over compact sets (up to a subsequence) to a function 
$u_{-1}=u_{-1}(x,y)$ and its $x-$derivative as $ a \to -1$ (which corresponds to $s \to 1$). 
We now choose the following 
test function: $\xi(x,y)=\eta_1(x) \eta_2(y)$, where $\eta_2(y)=1$ for $ 0 \leq y <1$ and 
$\eta_2(y)=0$ for $y >2$, whereas $\eta_1$ is any test function. 
We deduce   

\begin{equation*}
(1+a)\int_{\RR^{2}_+} y^{a} \Big \{\eta'_1(x)\eta_2(y)  \partial_x u_a  +  
\eta_1(x)  \eta'_2(y)   \partial_y u_a\Big \}dxdy 
\end{equation*}
\begin{equation*}
-\int_{\RR} c_af(u_a) \eta_1(x) dx=0. 
\end{equation*}

We now pass to the limit in each term. Thanks to the uniform bounds of Lemma \ref{regNL}, 
we have that 

$$\lim_{a \downarrow -1 }c_a\int_{\RR} f(u_a) \eta_1(x) \,dx $$
$$= \int_{\RR} f(u_{-1}(x,0)) \eta_1(x)\,dx. $$

Note  that the measure $(1+a) y^a \, dy $ is a probability 
measure on $(0,1)$ converging as $a \downarrow -1$ (in the weak$-*$ sense of measures)  to the Dirac 
measure $\delta_0$. More precisely,  given functions $w_a=w_a(y)$ continuous in $[0,\infty)$, 
with $|yw_a(y) | \leq C $ in $[0,+\infty)$ (with $C$ uniform in $a$) and 
with $w_a$ converging to a 
function $w_{-1}$ uniformly in compact sets of $[0,+\infty)$, then 
$$\lim_{a \downarrow -1} (1+a)\int_0^\infty y^a w_a(y)dy=w_{-1}(0).$$
Indeed, given $\varepsilon >0$ let $\delta >0$ such that $|w_a(x)-w_a(0)| \leq \varepsilon$ 
for all $x \in (0,\delta)$. Then, we write 
$$ (1+a)\int_0^\infty y^a w_a(y)dy= (1+a)\int_0^\delta y^a w_a(y)dy+ (1+a)\int_\delta^\infty 
y^a w_a(y)dy. $$
We have that 

$$(1+a) \int_0^\delta y^a(w_{-1}(0)+w_a(0)-w_{-1}(0))\,dy$$
tends to  $\lim_{a \downarrow -1}w_{-1}(0)\delta ^{1+a}=w_{-1}(0)$
and
$$(1+a) \int_0^\delta y^a|w_a(y)-w_{a}(0)|\,dy \leq \delta ^{1+a} \varepsilon.$$
Finally, 
\begin{equation}\label{intzero}
(1+a)\int_\delta^\infty y^a |w_a(y)|dy \leq (1+a)C \int_\delta ^\infty y^{a-1}�\,dy\leq 
C\frac{(1+a)}{-a}\delta^{a} \to 0
\end{equation}
as $a \downarrow -1$. This proves the claim above. 

We now divide the integral 
\begin{eqnarray*}
& & \ \hspace{-1.3cm} 
(1+a)\int_{\RR^{2}_+} y^{a} \Big \{\eta'_1(x) \eta_2(y)  \partial_x u_a +  
\eta_1(x)  \eta'_2(y) \partial_y u_a\Big \} dx dy\\
&=& (1+a)\int_{\RR \times (0,+\infty)} y^{a} \eta'_1(x)\eta_2(y)\partial_x u_a dxdy\\
& & 
+(1+a)\int_{\RR \times (1,2)} y^{a}   
\eta_1(x) \eta'_2(y)  \partial_y u_a dx dy .
\end{eqnarray*}
Thanks once again to Lemma \ref{regNL}, the observation above (with $w_a=\eta_2(\cdot)
\partial_x u_a(x,\cdot)$) and the gradient 
bounds of Lemma \ref{integrability}, we deduce 
\begin{eqnarray*}
& & \ \hspace{-2cm} \lim_{a \to -1}(1+a)\int_{\RR \times (0,+\infty)} y^{a} \eta'_1(x) 
\eta_2(y)\partial_x u_a \,dx dy\\
&=& \int_\RR u'_{-1} (x,0)\eta_1'(x)\,dx.
\end{eqnarray*}
By the same lemma, $|\partial_y u_a| \leq C y^{-1}$ uniformly in $a \in (-1,0)$, and thus 
the same computation as in \eqref{intzero} shows that 
$$\lim_{a \to -1}(1+a)\int_{\RR \times (1,2)} y^{a} \eta_1(x) \eta'_2(y) \partial_y 
u_a\,dxdy=0.$$

Therefore, the function $\overline v:=\overline v(x)=u_{-1}(x,0)$ satisfies  

\begin{equation}
\int_{\RR} \overline{v}'  \eta_1' -f(\overline v ) \eta_1 =0.
\end{equation}
Hence $\overline v $ is a weak solution of 
$$-\overline{v}'' =f(\overline v)\,\,\,\mbox{in } \RR ,$$
such that $\overline v(0)=0$ and $\overline{v}'\geq 0$ in $\RR$. 
As a consequence, the function $\overline v$ admits limits 
at $\pm \infty$, 
$$\lim_{x \rightarrow \pm \infty} \overline v(x)= L^\pm\,\,\,\,\in [-1,1] . $$

We now prove the convergence of the Hamiltonian, which will provide in addition  
that the function $\overline v $ is actually a layer, i.e. $L^\pm=\pm 1$. We apply the 
Hamiltonian estimate \eqref{modicatype} with $y=0$ to the layer $u_a$ for some $a \in (-1,1)$ 
with $G$ replaced by $c_aG.$ We deduce 
\begin{equation}\label{Gdouble}
0 < G-G(1)\,\,\,\,\mbox{in}\,\,\,(-1,1).
\end{equation}

Next, we integrate the equation satisfied by $\overline v$, we use the above observation 
now with $w_a(\cdot)=(\partial_x u_a)^2(x,\cdot)$ and we use the Hamiltonian identity 
\eqref{mmequaltemp} for the layer $u_a$ to obtain for all $x \in \RR$
\begin{equation}\label{modx}
G(\overline{v}(x))-G(L^+)=\frac12 (\overline{v}' )^2(x)=\lim_{a \downarrow -1}\frac{(1+a)}{2}
\int_0^{+\infty} y^a(\partial_x u_a)^2(x,y)dy
\end{equation}
$$=\lim_{a \downarrow -1}\Big \{ \frac{(1+a)}{2}\int_0^{+\infty} y^a(\partial_y u_a)^2(x,y)dy
+c_a G(u_{a}(x,0))-c_a G(1) \Big \},$$
and thus
\begin{equation}\label{conclus}
G(\overline v(x))-G(L^+) \geq \lim_{a \downarrow -1} c_a(G(u_{a}(x,0))-G(1))= G(\overline v(x))-G(1).
\end{equation}
Hence we have that 
$$G(L^+) \leq G(1), $$
that together with \eqref{Gdouble} and $L^+ \geq 0$ (since $\overline v(0)=0$) gives $L^+=1. $

In addition, we deduce that the inequality \eqref{conclus} must be an equality. Thus, the 
term that we have dropped to obtain the inequality must be zero, i.e.
\begin{equation}\label{mody}
\lim_{a \to -1}\frac{(1+a)}{2}\int_0^{+\infty} y^a(\partial_y u_a)^2(x,y)\,dy=0.
\end{equation}
In the same way, we prove that $L^-=-1. $
 Hence $\overline v $ is the layer solution connecting $-1$ to $1$, with $\overline v(0)=0$. 
The uniqueness of such $\overline v$ follows from the Hamiltonian equality \eqref{hamilteq}. 
\end{proof}

\section{Proof of Theorems \ref{necLayers}, \ref{modthm}, \ref{necBounds}, 
and \ref{modthmGS}}

We prove in this section the main theorems of our paper. They will follow
easily from our results in previous sections.

\begin{proof}[Proof of Theorem \ref{modthm}]
Part (i) follows from Lemmas \ref{hamilton} and \ref{modica1}. Part (ii) follows 
from Theorem \ref{asymps} and from \eqref{modx} and \eqref{mody} in its proof. 
\end{proof}

\begin{proof}[Proof of Theorem \ref{necLayers}]
If $v$ is a layer solution of \eqref{limits2}, its extension $u$ is a layer solution of 
\eqref{extAlpha} in $\RR^2_+$, up to a multiplicative constant in front of $f$ that tends 
to $1$ as $s \uparrow 1$. In part (i) of the theorem, \eqref{nec1} follows from \eqref{fpm}
in Lemma \ref{infinit}. The equality in \eqref{nec2}
is \eqref{Gequal} of 
Lemma \ref{hamilton}, while the inequality in  \eqref{nec2} follows from
taking $y=0$ in the statement of Lemma \ref{modica1}. Part (ii) of the theorem 
follows from  Theorem \ref{asymps}. 
\end{proof}

\begin{proof}[Proof of Theorem \ref{modthmGS}] 
It follows from Lemmas~\ref{haminf} and \ref{hamrad}.
\end{proof}

\begin{proof}[Proof of Theorem \ref{necBounds}] 
If $v$ is a radial solution in $\R^n$, its extension $u=u(x,y)$ belongs to
$L^{\infty}({\R^{n+1}_+})$ and it is a solution of \eqref{extAlpha}, up to a multiplicative 
positive constant in front of $f$. Clearly, $u$ is a radial solution in $x \in \RR^n$. 

The relation $f(0)=0$ is 
\eqref{fpmrad} in Lemma \ref{infinit}. 

The conclusion 
$G(0)>G(u(0,0))$ of the theorem follows from Lemma \ref{hamrad} and \eqref{derivradial}.
Indeed, in \eqref{hamradially} we let $r=0$ and later $r\to\infty$, and we use \eqref{lim0xmod} 
in Lemma~\ref{haminf} to obtain 
$$-G(v(0))\geq -(1+a) \int_0^\infty \frac{t^a}{2}u_y^2(0,t)\,dt -G(v(0)) \geq -G(0).$$ Thus 
$G(0)  \geq G(v(0)).$ But if $G(0)=G(v(0))$ then the function in \eqref{hamradially} would be 
constant in $r$. Hence, by \eqref{derivradial} and since $n >1$, $u_r(r,t) \equiv 0$ for all 
$r$ and $t$, and then $u$ is constant, contrary to our assumption.

It only remains to prove  the other statement of the theorem, 
$$f'(0)=-G''(0)\leq 0$$
under the assumption $u_r <0$. 
Without loss of generality, and to simplify notation, we may replace $f$ by a positive multiple of $f$  
in \eqref{extAlpha} and hence assume that $u$ solves
\begin{equation*}
\label{linr0}
\begin{cases}
\textrm{div\,} (y^{a}\,\nabla u)=0&\text{ in } \mathbb{R}^{n+1}_+\\ 
\displaystyle{\frac{\partial u}{\partial{\nu}^{a}}}
=f(u) &\text{ on } \partial\mathbb{R}^{n+1}_+=\RR^n .
\end{cases}
\end{equation*}
We differentiate both equations with respect to $r=|x|$, using that the first
one reads ${u}_{rr}+\frac{n-1}{r} {u}_{r}+ {u}_{yy}+\frac{a}{y} {u}_{y}=0$ in
$(0,+\infty) \times (0,+\infty)$. Let
$$
\psi:=-u_r >0 \qquad \text{in } (\R^n\setminus\{0\})\times (0,+\infty).
$$
We deduce that
\begin{equation*}
\label{linr1}
\textrm{div\,} (y^{a}\,\nabla \psi)= \frac{n-1}{|x|^2} y^a \psi
\qquad \text{in } (\R^n\setminus\{0\})\times (0,+\infty)
\end{equation*}
and
\begin{equation*}
\label{linr2}
f'(u) \psi= \frac{\partial \psi}{\partial{\nu}^{a}}
\qquad\text{for }y=0.
\end{equation*}

For $x_0 \in \R^n$,
let $\psi^{x_0}(x,y):=\psi(x-x_0,y)$, a positive function in 
$(\R^n\setminus\{x_0\})\times (0,+\infty)$. 
Let $u^{x_0}(x,y):=u(x-x_0,y)$. We have
\begin{equation}
\label{linr1t}
\textrm{div\,} (y^{a}\,\nabla \psi^{x_0})= \frac{n-1}{|x-x_0|^2} y^a
\psi^{x_0} \qquad \text{in } (\R^n\setminus\{x_0\})\times (0,+\infty)
\end{equation}
and
\begin{equation}
\label{linr2t}
f'(u^{x_0}) \psi^{x_0}= \frac{\partial \psi^{x_0}}{\partial{\nu}^{a}}
\qquad\text{for }y=0.
\end{equation}

For $R>0$, consider the cylinder
$\mathcal C_{R}= \Gamma_R^0 \times (0,R) \subset \RR^{n+1}_+$, where $\Gamma_R^0$ 
is the ball of center $0$ and radius $R$ in $\RR^n$.
Let $\xi$ be any $C^1$ function in $\overline{\mathcal C_{R}}$ vanishing on
$\{|x|=R\} \times [0,R)$ and on $\Gamma_R^0 \times \{y=R\}$.

For $|x_0|>R$, we multiply \eqref{linr2t} by $\xi^2/\psi^{x_0}$ ---note that
$\psi^{x_0} >0$ in  $\Gamma_R^0 \times [0,R)$---, we integrate in  $\Gamma_R^0$
and use \eqref{linr1t} to obtain
\begin{eqnarray*}
& &  \hspace{-1.2cm}
\int_{\Gamma^0_R} f'(u^{x_0}) \xi^2 
\\
&=&
\int_{\mathcal C_{R}} \textrm{div\,} (y^{a}\,\nabla \psi^{x_0}) 
\frac{\xi^2}{\psi^{x_0}} + y^a\nabla \psi^{x_0}\cdot\nabla\frac{\xi^2}{\psi^{x_0}}
\\
&=&
\int_{\mathcal C_{R}} \frac{n-1}{|x-x_0|^2} y^a \xi^2
+ y^a\nabla \psi^{x_0}\cdot\nabla\frac{\xi^2}{\psi^{x_0}}
\\
&=&
\int_{\mathcal C_{R}} \frac{n-1}{|x-x_0|^2} y^a \xi^2
+ y^a\left\{2 \xi\nabla \xi\cdot\frac{\nabla\psi^{x_0}}{\psi^{x_0}}
-\xi^2 \frac{|\nabla\psi^{x_0}|^2}{(\psi^{x_0})^2}\right\}
\\
&\leq&
\int_{\mathcal C_{R}} \frac{n-1}{|x-x_0|^2} y^a \xi^2
+ y^a|\nabla\xi|^2,
\end{eqnarray*}
where we have used Cauchy-Schwarz inequality.

Letting $|x_0|\to \infty$, we deduce
\begin{equation}
\label{f0eig}
f'(0)\leq \frac{\int_{\mathcal C_{R}}y^a|\nabla\xi|^2}{\int_{\Gamma^0_R} \xi^2}
\end{equation}
for any $C^1$ function $\xi$ in $\overline{\mathcal C_{R}}$ vanishing on
$\{|x|=R\} \times [0,R)$ and on $\Gamma_R^0 \times \{y=R\}$.

Let $\varphi_R=\varphi_R(x)>0$ be the first eigenfunction 
of $-\Delta_x$ in $\Gamma_R^0$ with Dirichlet boundary conditions, i.e., 
\begin{equation*}
\begin{cases}
-\Delta_x \varphi_R=\lambda_R \varphi_R & \mbox{in } \Gamma_R^0\\ 
\varphi_R=0 & \mbox{on } \partial \Gamma_R^0,
\end{cases}
\end{equation*}
where $\lambda_R=c(n)/R^2>0$ is the first Dirichlet eigenvalue
of $-\Delta_x$ in the ball $\Gamma_R^0$.
Let $h_R=h_R(y)\in [0,1]$ be a smooth function with compact support in $[0,R)$ such that
$h_R\equiv 1$ in $[0,R/2]$ and $|h_R'|\leq C/R$ for some constant $C$.
Take 
$$
\xi=\xi_{R}=\varphi_{R}(x) h_R(y)
$$ 
in \eqref{f0eig}. We have
\begin{eqnarray*}
\int_{\mathcal C_{R}}y^a|\nabla\xi_R|^2 & = &
\int_{\mathcal C_{R}}y^a\left\{|\nabla_x\varphi_R|^2 h_R^2+
\varphi_R^2 (h_R')^2\right\} \\
&=&
\left\{ \lambda_R \int_0^R y^ah_R^2 dy + \int_0^R y^a(h_R')^2 dy
\right\} \int_{\Gamma^0_R} \varphi_R^2\\
&=&
\left\{ CR^{-2} \int_0^R y^ah_R^2 dy + \int_0^R y^a(h_R')^2 dy
\right\}
\int_{\Gamma^0_R} \xi_R^2
\\
&\leq &
CR^{-2} \int_0^R y^a dy 
\int_{\Gamma^0_R} \xi_R^2 = CR^{-2}\frac{R^{2-2s}}{2-2s} 
\int_{\Gamma^0_R} \xi_R^2
\\
&= &\frac{C}{2-2s} R^{-2s} 
\int_{\Gamma^0_R} \xi_R^2.
\end{eqnarray*}
Using this in \eqref{f0eig} and letting $R\to\infty$, we conclude
that $f'(0)\leq 0$, as claimed.
\end{proof}

\bibliographystyle{plain} 
\bibliography{bibliofile_A_new}

\end{document}